\documentclass{amsart}
\usepackage{amssymb}
\usepackage{amsrefs}
\usepackage{graphicx}
\usepackage{tikz-cd}
\usepackage[percent]{overpic}
\usepackage[shortlabels]{enumitem}
\usepackage{mathtools}
\usepackage{verbatim}
\usepackage{hyperref}

\usepackage{enumitem}
\usepackage{caption}
\usepackage{subcaption}

\usepackage{tikz}
\usepackage{tkz-euclide}
\usetikzlibrary{arrows}
\usetikzlibrary{patterns}

\newcommand\customticknum{3}   
\newcommand\customticklen{0.5em} 

\makeatletter
\newcommand*{\mint}[1]{%
  \mint@l{#1}{}%
}
\newcommand*{\mint@l}[2]{%
  \@ifnextchar\limits{%
    \mint@l{#1}%
  }{%
    \@ifnextchar\nolimits{%
      \mint@l{#1}%
    }{%
      \@ifnextchar\displaylimits{%
        \mint@l{#1}%
      }{%
        \mint@s{#2}{#1}%
      }%
    }%
  }%
}
\newcommand*{\mint@s}[2]{%
  \@ifnextchar_{%
    \mint@sub{#1}{#2}%
  }{%
    \@ifnextchar^{%
      \mint@sup{#1}{#2}%
    }{%
      \mint@{#1}{#2}{}{}%
    }%
  }%
}
\def\mint@sub#1#2_#3{%
  \@ifnextchar^{%
    \mint@sub@sup{#1}{#2}{#3}%
  }{%
    \mint@{#1}{#2}{#3}{}%
  }%
}
\def\mint@sup#1#2^#3{%
  \@ifnextchar_{%
    \mint@sup@sub{#1}{#2}{#3}%
  }{%
    \mint@{#1}{#2}{}{#3}%
  }%
}
\def\mint@sub@sup#1#2#3^#4{%
  \mint@{#1}{#2}{#3}{#4}%
}
\def\mint@sup@sub#1#2#3_#4{%
  \mint@{#1}{#2}{#4}{#3}%
}
\newcommand*{\mint@}[4]{%
  \mathop{}%
  \mkern-\thinmuskip
  \mathchoice{%
    \mint@@{#1}{#2}{#3}{#4}%
        \displaystyle\textstyle\scriptstyle
  }{%
    \mint@@{#1}{#2}{#3}{#4}%
        \textstyle\scriptstyle\scriptstyle
  }{%
    \mint@@{#1}{#2}{#3}{#4}%
        \scriptstyle\scriptscriptstyle\scriptscriptstyle
  }{%
    \mint@@{#1}{#2}{#3}{#4}%
        \scriptscriptstyle\scriptscriptstyle\scriptscriptstyle
  }%
  \mkern-\thinmuskip
  \int#1%
  \ifx\\#3\\\else_{#3}\fi
  \ifx\\#4\\\else^{#4}\fi  
}
\newcommand*{\mint@@}[7]{%
  \begingroup
    \sbox0{$#5\int\m@th$}%
    \sbox2{$#5\int_{}\m@th$}%
    \dimen2=\wd0 %
    \let\mint@limits=#1\relax
    \ifx\mint@limits\relax
      \sbox4{$#5\int_{\kern1sp}^{\kern1sp}\m@th$}%
      \ifdim\wd4>\wd2 %
        \let\mint@limits=\nolimits
      \else
        \let\mint@limits=\limits
      \fi
    \fi
    \ifx\mint@limits\displaylimits
      \ifx#5\displaystyle
        \let\mint@limits=\limits
      \fi
    \fi
    \ifx\mint@limits\limits
      \sbox0{$#7#3\m@th$}%
      \sbox2{$#7#4\m@th$}%
      \ifdim\wd0>\dimen2 %
        \dimen2=\wd0 %
      \fi
      \ifdim\wd2>\dimen2 %
        \dimen2=\wd2 %
      \fi
    \fi
    \rlap{%
      $#5%
        \vcenter{%
          \hbox to\dimen2{%
            \hss
            $#6{#2}\m@th$%
            \hss
          }%
        }%
      $%
    }%
  \endgroup
}

\newcommand{\x}{\scalebox{1.2}{$\chi$} } 

\newcommand{\br}{\overline}
\newcommand{\R}{\mathbb R}
\newcommand{\C}{\mathbb C}
\newcommand{\D}{\mathbb D}
\newcommand{\Z}{\mathbb Z}
\newcommand{\N}{\mathbb N}

\newcommand{\p}{\mathbb P}

\theoremstyle{plain}
\newtheorem{theorem}{Theorem}
\newtheorem{lemma}[theorem]{Lemma}

\newtheorem{prop}[theorem]{Proposition}
\newtheorem{corollary}[theorem]{Corollary}

\newtheoremstyle{case}{}{}{}{}{}{:}{ }{}
\theoremstyle{case}

\theoremstyle{definition}
\newtheorem{definition}[theorem]{Definition}

\theoremstyle{remark}
\newtheorem{remark}[theorem]{Remark}
\newtheorem{question}[theorem]{Question}

\DeclareMathOperator{\dist}{\textup{\text{dist}}}
\DeclareMathOperator{\diam}{\textup{\text{diam}}}

\DeclareMathOperator{\inter}{\textup{\text{int}}}

\DeclareMathOperator{\length}{\textup{\text{length}}}

\DeclareMathOperator{\loc}{\textup{loc}}

\numberwithin{equation}{section}
\numberwithin{theorem}{section}

\begin{document}

\title{A removability theorem for Sobolev functions and Detour sets}
\author{Dimitrios Ntalampekos}
\thanks{The author was partially supported by NSF grant DMS-1506099}
\address{Institute for Mathematical Sciences\\ Stony Brook University, Stony Brook\\
NY 11794-3651, USA}
\email{dimitrios.ntalampekos@stonybrook.edu}
\date{\today}

\maketitle

\begin{abstract}
We study the removability of compact sets for continuous Sobolev functions. In particular, we focus on sets with infinitely many complementary components, called ``detour sets", which resemble the Sierpi\'nski gasket. The main theorem is that if $K\subset \R^n$ is a detour set and its complementary components are sufficiently regular, then $K$ is $W^{1,p}$-removable for $p>n$. Several examples and constructions of sets where the theorem applies are given, including the Sierpi\'nski gasket, Apollonian gaskets, and Julia sets.
\end{abstract}

\section{Introduction}

\subsection{Background}
In this paper we study the removability of sets for Sobolev functions in $\R^n$. The problem originally arises from the problem of removabilty of sets for (quasi)conformal maps.

\begin{definition}\label{Conformal removability}
A compact set $K\subset U\subset \R^2$ is \textit{(quasi)conformally removable inside the domain $U$} if any  homeomorphism of $U$, which is (quasi)conformal on $U\setminus K$, is (quasi)conformal on $U$. 
\end{definition}

(Quasi)conformal removability of sets is of particular interest in Complex Dynamics, since it provides a tool for upgrading a topological conjugacy between two dynamical systems to a (quasi)conformal conjugacy. Another application of removability results is in the problem \textit{conformal welding}. The relevant result is that if a Jordan curve $\partial \Omega$ is (quasi)conformally removable, then the welding map that arises from $\partial \Omega$ is unique, up to M\"obius transformations. However, the converse is not known, in case $\partial \Omega$ has zero area. We direct the reader to \cite{Yo}, and the references therein, for a comprehensive survey on the topic.

A stronger notion of removability is the notion of \textit{Sobolev $W^{1,2}$-removability}. Recall that for an open set $U\subset \R^n$ and $p\geq 1$ we say that a function $f$ lies in the Sobolev space $W^{1,p}(U)$ if $f\in L^p(U)$ and there exist functions $\partial_i f\in L^p(U)$, $i=1,\dots,n$, such that for every smooth function $\phi\colon \R^n\to \R$ with compact support in $U$, we have
\begin{align}\label{Intro Definition Sobolev}
\int_U \partial_i f \cdot  \phi=- \int_U f\cdot \partial_i \phi  
\end{align}
for all $i=1,\dots,n$, where $\partial_i\phi= \frac{\partial\phi}{\partial x_i}$ denote the partial derivatives of $\phi$ in the coordinate directions of $\R^n$. We give a preliminary definition of $W^{1,p}$-removability in $\R^n$:

\begin{flushleft}
\textit{We say that a compact set $K\subset U \subset \R^n$ is \textit{$W^{1,p}$-removable inside a domain $U$}, if any function that is continuous in $U$ and lies in $W^{1,p}(U\setminus K)$, also lies in $W^{1,p}(U)$. In other words, $W^{1,p}(U\setminus K)\cap C^0(U)= W^{1,p}(U)\cap C^0(U)$ as sets.}
\end{flushleft}
Here, $C^0(U)$ denotes the space of continuous functions on $U$. It is immediate to check from the definition that removability is a local property, namely $K$ is $W^{1,p}$-removable inside $U$ if and only if for each $x\in K$ there exists $r>0$ such that $W^{1,p}(B(x,r)\setminus K) \cap C^0(B(x,r))= W^{1,p}(B(x,r))\cap C^0(B(x,r))$. Indeed, for the sufficiency of the latter condition one can argue using a partition of unity. For the necessity let $B(x,r)\subset U$ and note that if $u\in W^{1,p}(B(x,r)\setminus K) \cap C^0(B(x,r))$ and $\phi$ is smooth with compact support inside $B(x,r)$, then we may consider a smooth function $\psi$ that is compactly supported in $B(x,r)$ and is identically equal to $1$ on the support of $\phi$. Since $u\cdot \psi\in W^{1,p}(U\setminus K)\cap C^0(U)= W^{1,p}(U)\cap C^0(U)$, we then have $\int \partial_i u \cdot \phi= \int \partial_i (u\psi)\cdot \phi= - \int  u\psi \partial_i \phi = -\int u \partial_i \phi$. This shows that $u\in W^{1,p}(B(x,r))\cap C^0(B(x,r))$.

Hence, our definition of removability can be simplified to the following:

\begin{definition}\label{Intro Def Removability}
Let $p\geq 1$. We say that a compact set $K\subset \R^n$ is $W^{1,p}$-removable if any function that is continuous in $\R^n$ and lies in $W^{1,p}(\R^n\setminus K)$, also lies in $W^{1,p}(\R^n)$. In other words,
\begin{align*}
W^{1,p}(\R^n\setminus K) \cap C^0(\R^n)= W^{1,p}(\R^n)\cap C^0(\R^n).
\end{align*}
\end{definition}

We remark that a similar simplification is also possible for the problem of (quasi)con\-formal removability in the plane. In the definition we are looking at \textit{continuous} functions on all of $\R^n$. It is not known whether $W^{1,2}$-removability in the plane and (quasi)conformal removability are equivalent. However, so far, the techniques for proving removability of a set in the two distinct settings do not differ.  

A similar problem has been studied by Koskela \cite{Ko}, but the definition of a removable set is slightly different, since no continuity is assumed for the functions; thus, it is a stronger notion of removability. In fact, in our case, the continuity will be crucially used in the proofs.

It is almost immediate from the definition that $W^{1,p}$-removability of a set $K$ of measure zero implies $W^{1,q}$-removability for $q>p$. Indeed, this follows immediately from H\"older's inequality, and the locality of the definition of a Sobolev function.

It can be shown that a set $K\subset \R^n$ having $\sigma$-finite Hausdorff $(n-1)$-measure is  $W^{1,p}$-removable for all $p\geq 1$. The proof is rather a modification of the classical proof that such sets are removable for quasiconformal maps in $\R^n$; see \cite[Section 35]{Va}.

If a compact set $K\subset \R^2$ has area zero and is non-removable for (quasi)conformal maps, then (by definition) there exists a homeomorphism $f\colon \R^2\to \R^2$ that is (quasi)conformal in $\R^2 \setminus K$, but it is not (quasi)conformal on $\R^2$. In particular, $f\in W^{1,2}(B(0,R) \setminus K)$, but $f\notin W^{1,2}(B(0,R))$ for some large ball $B(0,R)\supset K$. This implies that $K$ is non-removable for $W^{1,2}$, and, thus, non-removable for $W^{1,p}$, $1\leq p\leq 2$, by our earlier remarks. Hence, there is a pool of examples of non-removable sets for Sobolev functions, including all sets of area zero that are non-removable for (quasi)conformal maps. We cite the \textit{flexible curves} of Bishop \cite{Bi}, and the non-removable \textit{graphs} of Kaufman \cite{Kau}.

\begin{question}Are there sets $K\subset \R^n$ which are $W^{1,p}$-removable for some exponents $p$, and non-removable for others?
\end{question}

The answer to this question is yes. In \cite{KRZ} an example is given of a Cantor set $E\subset \R^2$ that is  $W^{1,q}$-removable for some exponents $1<q<\infty$, and non-removable for other exponents. The particular statement is the following \cite[Lemma 4.4]{KRZ}: For each $p>2$ there exists a Cantor set $E\subset \R^2$ of area zero that is $W^{1,q}$-removable for $q>p$, but it is non-removable for $1\leq q\leq p$. We remark that in the proof of the second part, the authors construct a function $u\in W^{1,q} (\R^2\setminus E)$ that does not lie in $W^{1,q}(\R^2)$. In fact, the function that they construct is continuous on $\R^2$, and this shows the non-removability in our sense.

Another natural question is the following:
\begin{question}\label{Question positive area} If $K\subset \R^n$ has positive Lebesgue measure, then is it true that it is non-remov\-able for \textit{all} spaces $W^{1,p}$, $1\leq p <\infty$ ? 
\end{question}
The answer is trivially yes if $K$ has non-empty interior, since we can consider a continuous function with no partial derivatives, compactly supported in an open ball in $\inter(K)$. However, we have not been able to locate an answer to this question in the literature.

Some fairly general classes of sets were proved to be removable in \cite{Jo} and \cite{JS}, and the results of the latter were generalized in \cite{KN} in the planar case. In \cite{Jo} it is proved that boundaries of planar \textit{John domains} are $W^{1,2}$-removable. In the subsequent paper \cite{JS} this result was improved, and certain \textit{quasihyperbolic boundary conditions} on an open set $\Omega\subset \R^n$ were found to imply $W^{1,n}$-removability for the boundary $\partial \Omega$; see \cite[Theorem 1]{JS}. Their precise condition is the condition in the conclusion of our Lemma \ref{Shadow-QH distance}. Also, one obtains the $W^{1,n}$-removability conclusion for the boundary of the union of \textit{finitely many} open sets $\Omega\subset \R^n$ if they satisfy the quasihyperbolic boundary condition; see  \cite[p. 265]{JS}. 

However, so far, no general removability result was known for sets $K\subset \R^n$ whose complement has \textit{infinitely many}  components. One classical example of a non-removable such set is the \textit{standard Sierpi\'nski carpet} $S$; see Figure \ref{fig:carpet-Apollonian}. This is constructed as follows. We subdivide the unit square of $\R^2$ into nine squares of sidelength $1/3$, and then remove the middle open square. In the next step, we subdivide each of the remaining eight squares into nine squares of sidelength $1/9$, and remove the middle squares. Then one proceeds inductively, and the remaining compact set $S$ is the standard Sierpin\'ski carpet. Note that $S$ contains a copy of $K=C\times [0,1]$, where $C$ denotes the middle-thirds Cantor set. Let $h\colon \R\to \R$ be the Cantor staircase function and let $\psi\colon \R\to [0,1]$ be a smooth function with $\psi\equiv 0$ outside $[0,1]$, and $\psi\equiv 1$ in $[1/9, 8/9]$. Then $f_S(x,y)\coloneqq x+ h(x)\psi(y)$ is a continuous function that lies in $W^{1,p}(U\setminus S)$ for all $p\geq 1$, but not in $W^{1,p}(U)$, where $U$ is a bounded open set containing $S$. In fact, the Sierpi\'nski carpet is not either quasiconformally removable since  the map $(x,y)\mapsto (x+h(x)\psi(y),y)$ is a homeomorphism of $\R^2$ that is quasiconformal in $\R^2\setminus S$, but not in $\R^2$.

On the other hand, prior to this work, there has been no conclusion regarding the removability of the \textit{Sierpi\'nski gasket}; Figure \ref{fig:gasket}. To construct the latter, we subdivide an equilateral triangle of sidelength $1$ into four equilateral triangles of sidelength $1/2$, and then remove the middle open triangle. In the next step we subdivide each of the remaining three triangles into four equilateral triangles of sidelength $1/4$, and then remove the middle triangles. After proceeding inductively, the remaining compact set $K$ is the Sierpi\'nski gasket.  The difference with the Sierp\'inski carpet, which gives some ``hope" towards proving removability, is that the closures of some  complementary components  of the gasket meet each other, in contrast to the complementary components of the carpet, which are all disjoint squares. 

The (quasi)conformal removability of sets that resemble the Sierpi\'nski gasket would be of great interest, not only in Complex Dynamics, but also in the theory of the \textit{Schramm-Loewner Evolution}, since it would provide insight to a question raised by Scott Sheffield regarding the removability of the trace of SLE$_\kappa$ for $\kappa\in (4,8)$; see \cite{Sh}.

\subsection{Main result}
\begin{figure}
	\centering
	\input{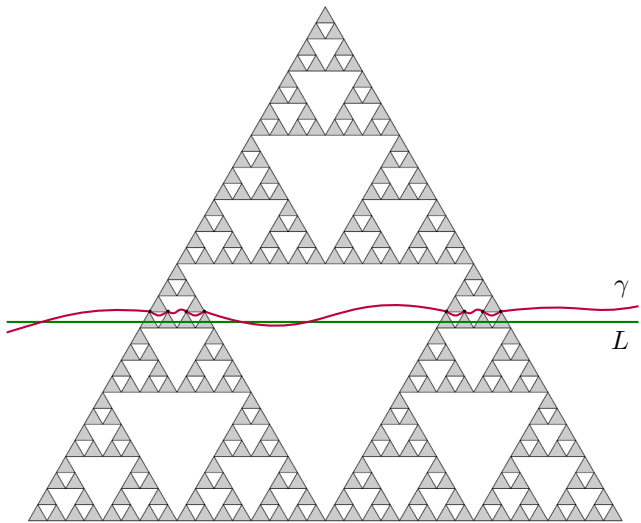}	
	\caption{The Sierpi\'nski gasket, and a detour path $\gamma$ near the line $L$.} \label{fig:gasket}
\end{figure}

In this work we focus on proving a removability result for sets resembling the Sierpi\'nski gasket. The property of the gasket $K$ that we are interested in is the following:
\begin{quote}
\textit{For ``almost every" line $L$ intersecting $K$, and for every $\varepsilon>0$ there exists a ``detour path" $\gamma$ that $\varepsilon$-follows the line $L$, but intersects only finitely many complementary components of $K$.}
\end{quote}
In other words, we can ``travel" in the direction of the line $L$ using only finitely many components in the complement of $K$, but still staying arbitrarily close to the line $L$; see Figure \ref{fig:gasket}. We call such sets \textit{detour sets}, and their precise definition is given in Section \ref{Section Detour}, where an extra technical condition is added.

In order to formulate the theorem, we also need to impose some regularity on the boundaries of the complementary components of $K$; this is necessary, since there are non-removable Jordan curves. A fairly general condition that is easy to check in applications is the requirement that the complementary components of $K$ are  \textit{uniform H\"older domains}. 
 
A simply connected planar domain $\Omega$ is a H\"older domain if there exists a conformal map from the unit disk $\D$ onto $\Omega$ that extends to be H\"older continuous on $\br \D$. Note that there are some implicit constants contained in the statement. The general definition of a H\"older domain in $\R^n$ is given in Section \ref{Section Holder}. The complementary components of a compact set $K\subset \R^n$ are \textit{uniform} H\"older domains if they are H\"older domains and the implicit constants are the same for all of them; cf. Definition \ref{Main Holder Detour}. We now state the Main Theorem:

\begin{theorem}\label{Intro main theorem}
Let $K\subset \R^n$ be a detour set whose complementary components are uniform H\"older domains. Then $K$ is $W^{1,p}$-removable for $p>n$.
\end{theorem}

The method used in the proof relies partly on the result of Jones and Smirnov in \cite{JS}, and partly on the detour property. The first result tells us that the boundaries of the complementary components of $K$ can be essentially ignored, since they are already $W^{1,p}$-removable for $p\geq n$. The detour property is used to deal with the ``hidden" parts of $K$ that do not lie on the boundary of any complementary component. Our proof in this part is what imposes the assumption $p>n$. Unfortunately there seems to be no immediate generalization of the proof which would yield $W^{1,n}$-removability, and thus (quasi)conformal removability; nevertheless, see the \textit{Update} below, which shows that this is not expected.

However, if $K$ is the Sierpi\'nski gasket, this theorem implies the following: in case there exists an example of a globally continuous function $f\in W^{1,2}(\R^2\setminus K)$ which does not lie in $W^{1,2}(\R^2)$, then one necessarily has 
\begin{align*}
 \int |\nabla f |^p =\infty
\end{align*}
for \textit{all} $p>2$. See also the remarks in Section \ref{Section Concluding}. This is another difference between the gasket and the carpet $S$. Indeed, by modifying the function $f_S$ that we gave earlier, one may obtain a globally continuous function $f\in W^{1,2}(\R^2\setminus S)$, with $\int |\nabla f|^p<\infty$ for all $p\geq 1$, but $f\notin W^{1,p}(\R^2)$.

\begin{figure}
\centering
		\begin{tikzpicture}
			\node[] (carpet) at (0,0)
				{\includegraphics[width=.45\textwidth]{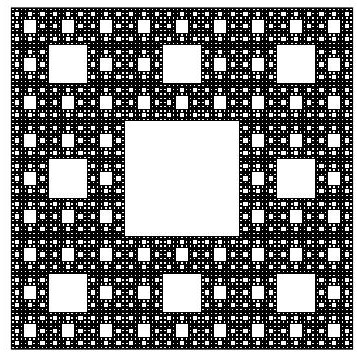}};
			\node[] (apollonian) at (6,0)
				{\includegraphics[width=.45\textwidth]{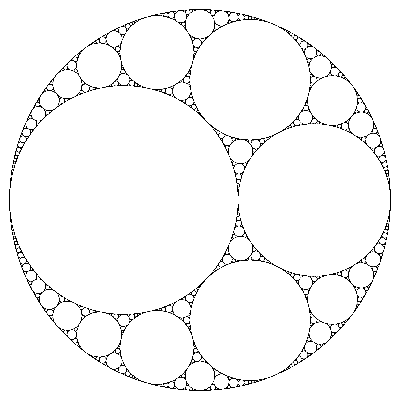}};
		\end{tikzpicture}
		\caption{The standard Sierpi\'nski carpet (left) and an Apollonian gasket (right). \label{fig:carpet-Apollonian}}
\end{figure}

It is remarkable that the detour property, as well as the uniform H\"older domain property, are easy to check in several circumstances, thus yielding:

\begin{corollary}\label{Intro Corollary}
The following planar sets are $W^{1,p}$-removable for $p>2$:
\begin{enumerate}[\upshape(i)]
\item the Sierpi\'nski gasket,
\item Apollonian gaskets,
\item generalized Sierpi\'nski gasket Julia sets of postcritically finite rational maps (under some combinatorial assumptions).
\end{enumerate}
\end{corollary}
Definitions of the latter two classes of sets are given in Section \ref{Section Examples}. We also include there a general construction of detour sets in the plane.

\subsection*{Update}
Since the completion and distribution of the first version of this paper, there has been some further progress. The current author in \cite{Nta} has proved that the Sierpi\'nski gasket is non-removable for (quasi)conformal maps and for $W^{1,p}$ functions, for $1\leq p\leq 2$. Combined with Corollary \ref{Intro Corollary} above, this shows that the gasket is $W^{1,p}$-removable if and only if $p>2$, so the result of the current paper is sharp. Moreover, it is proved in \cite{Nta} that not only the Sierpi\'nski gasket, but also all homeomorphic copies of the gasket are non-removable for $W^{1,p}$ functions, for $1\leq p\leq 2$; note that these are also detour sets, by Remark \ref{Example Homeo invariance} in the current paper. Furthermore, in the same paper a positive answer to Question \ref{Question positive area} is given.

Another related result that has been recently proved by the current author in \cite{NtaCarpet} is that all Sierpi\'nski carpets, i.e., homeomorphic copies of the standard Sierpi\'nski carpet, are non-removable for (quasi)conformal maps, and thus for $W^{1,p}$ functions, for $1\leq p\leq 2$. For the higher-dimensional version of this result, see \cite{NtaWu}.

\subsection{Organization of the paper}
	
In Section \ref{Section Notation} we introduce our notations and terminology. 

In Section \ref{Section Detour} the detailed definition of detour sets is given and several properties are established for these sets. In particular, in the planar case, the topology allows us to determine the geometry of detour sets very well; see Proposition \ref{Detour dimension 2}. We also introduce conditions in Proposition \ref{Detour Zero measure}, which ensure that a detour set has measure zero in $\R^n$.

Section \ref{Section Holder} includes the definition and several properties of H\"older domains in $\R^n$. We reprove some of these properties, since we are interested in the dependence of some inequalities on the constants related to a H\"older domain. 

In Section \ref{Section Sobolev} we prove some classical inequalities for Sobolev functions, whose domain is a H\"older domain. Again, we are interested in the dependence on the implicit constants, especially in Proposition \ref{Sobolev Holder estimate}. We also  state here a variant, or rather a special case, of \cite[Theorem 1]{JS}, including a proof for the sake of completeness; see Proposition \ref{Sobolev Holder boundary}.

Section \ref{Section Main Theorem} is occupied by the proof of Theorem \ref{Intro main theorem}.

Section \ref{Section Examples} is divided in two parts. In the first part we give two general constructions of detour sets, and in the second part we prove Corollary \ref{Intro Corollary}, including also the definition of Apollonian gaskets and generalized Sierpi\'nski gasket Julia sets.

Finally, in Section \ref{Section Concluding} we discuss how our results could give some insight towards proving the (non-)removability of the Sierpi\'nski gasket.

\subsection*{Acknowledgments}
The author would like to thank Mario Bonk for introducing him to the problem of removability, and for several fruitful discussions and explanations on the background of the problem and the proofs of previous results. Additional thanks go to Huy Tran for pointing out the connection of the problem to SLE, to Ville Tengvall for pointing out the reference \cite{KRZ}, and to Pekka Koskela for a motivating discussion. The author is also grateful to Vasiliki Evdoridou, Malik Younsi, and the anonymous referee for their comments and corrections.

This paper was written while the author was visiting University of Helsinki. He thanks the faculty and staff of the Department of Mathematics at the University of Helsinki for their hospitality.

\section{Notation and Terminology}\label{Section Notation}
A \textit{curve} or \textit{path} $\gamma$ is a continuous function $\gamma\colon I\to \R^n$, where $I\subset \R$ is an interval. We will also denote by $\gamma \subset \R^n$ the trace of the curve $\gamma$, i.e., the set $\gamma(I)$. An \textit{open path} $\gamma$ is a continuous function $\gamma\colon (0,1)\to \R^n$ that is the restriction of a path $\br \gamma\colon [0,1]\to\R^n$. In other words, an open path has endpoints. A \textit{simple} path is an injective path. 

A \textit{Jordan region} $\Omega \subset \R^2$ is an open set whose boundary $\partial \Omega$ is a Jordan curve. We call $\Omega \subset \R^2$ an \textit{unbounded Jordan region}, if $\partial \Omega$ is a Jordan curve but $\Omega$ is the unbounded component of $\R^2 \setminus \partial \Omega$. 

All distances will be with respect to the Euclidean distance of $\R^n$. We use the standard open ball notation $B(x,r)=\{y\in \R^n: |x-y|<r\}$, and we use the notation $\br B(x,r)$ for the closed ball.

We denote by $\widehat{\C}$ the Riemann sphere with the standard topology. 

The $n$-dimensional Lebesgue measure is denoted by $m_n$. If $f$ is an integrable function on $\R^n$ we will write $\int f$ for its integral. Only in certain situations, to avoid confusion, we will write instead $\int f(x)\, dm_n(x)$. 

We use the notation $a\lesssim b$ if there exists an implicit constant $C>0$ such that $a\leq C b$, and $a\simeq b$ if there exists a constant $C>0$ such that $C^{-1}b\leq a\leq Cb$. We will be mentioning  the parameters on which the implicit constant $C$ depends, if they are not universal constants.

\section{Detour sets}\label{Section Detour}

\begin{definition}\label{Definition Detour set}
Let $K\subset \R^n$, $n\geq 2$, be a compact set and assume that it has infinitely many complementary components, denoted by $\bigcup_{k=0}^\infty \Omega_k =\R^n\setminus K$. The set $K$ is a \textit{detour set} if the following holds: 
\newline
There exist $n$ linearly independent directions such that almost every line $L$, parallel to one of these directions, has the property that for all $\varepsilon>0$ there exists a path $\gamma$ lying  $\varepsilon$-close  to the line $L$ in the Hausdorff sense, such that:
\begin{enumerate}[(i)]
\item $\gamma \subset \bigcup_{k=0}^\infty \br \Omega_k$,
\item there is a finite set $\Gamma \subset \{k: \gamma \cap \br \Omega_k\neq \emptyset\}$ such that $\gamma \subset \bigcup_{k\in \Gamma} \br \Omega_k$, and
\item $\Gamma \subset \{k: L \cap \br \Omega_k\neq \emptyset\}$.
\end{enumerate}  
We say that a line $L$ satisfying the above has the \textit{detour property}, and a path $\gamma$ as above is called a \textit{detour path}.
\end{definition}

In other words, we require that a detour path $\gamma$ $\varepsilon$-follow the line $L$ and be contained in the union of only finitely many components $\br \Omega_k$, all of which are intersected by $L$. 

In what follows $\Omega_0$ will always denote the unbounded component of $\R^n\setminus K$, and $\Omega_k$, $k\geq 1$, will be the bounded components.

We record an easy observation in the next lemma.

\begin{lemma}\label{Detour empty interior}
Let $K\subset \R^n$ be a detour set. Then $\inter(K)=\emptyset$.
\end{lemma}
\begin{proof}
If there exists a ball $B\subset K$, then we can find a non-exceptional line $L$ with the detour property, passing through $B$. A sufficiently small neighborhood $U$ of $L$ is separated by $B$. Then any path  $\gamma \subset U$ that $\varepsilon$-follows the line $L$ has to intersect $B$. This contradicts the detour property of $L$, since a detour path $\gamma$ has to lie entirely in $\bigcup_{k=0}^\infty \br \Omega_k$.
\end{proof}

In dimension $2$ we obtain some interesting consequences. 

A point $x$ in a metric space $X$ is a \textit{local cut point} if there exists a connected open neighborhood $U$ of $x$ such that $U\setminus \{x\}$ is disconnected.

\begin{prop}\label{Detour dimension 2}
Let $K\subset \R^2$ be a detour set, whose complementary components $\Omega_k$, $k\in \N\cup \{0\}$, are Jordan regions, regarded as subsets of $\widehat \C$. Then
\begin{enumerate}[\upshape(i)]
\item $K$ is connected,
\item $\diam(\Omega_k) \to 0$ as $k\to\infty$,
\item $K$ is locally connected, and
\item $K$ has a dense set of local cut points.
\end{enumerate}
\end{prop}
\begin{proof}
For (i) we do not need to use the detour property. For $n\in \N\cup\{0\}$ let $K_n= \R^2 \setminus \bigcup_{k=0}^n\Omega_k$ and note that $K_n$ is connected. Indeed, if $x,y\in K_n$ and $L$ is the line segment connecting $x$ and $y$, then the union of the set $L\cap K_n$ with the sets $\partial \Omega_k$ that intersect $L$, $k=0,\dots,n$, is a continuum in $K_n$ connecting $x$ and $y$; it is important here that $\partial \Omega_k$ is connected. Then $K= \bigcap_{n=0}^\infty K_n$ is connected, since it is the decreasing intersection of continua; see \cite[Theorem 5.3, Chapter IV.5, p.~81]{New}.

For (ii) we argue by contradiction, assuming that there exists $\delta>0$ such that the set $Z\coloneqq \{k:\diam(\Omega_k) \geq \delta\}$ is infinite. We first show that there exists a set of lines with positive measure, parallel to one of the two directions given by the definition of a detour set, that intersect $\Omega_k$ for infinitely many $k\in Z$.

Denote by $v_0,w_0 \in \R^2$ the two linearly independent directions as in the definition of a detour set, and the corresponding lines by $L_{v_0}$ and $L_{w_0}$. Also, for $v\in \{v_0\}^\perp$ and $w\in \{w_0\}^\perp$ define $L_v\coloneqq L_{v_0} +v$, and $L_{w}\coloneqq L_{w_0}+w$, which are parallel translations of $L_{v_0}$ and $L_{w_0}$, respectively. If $\Omega_k$ is a bounded component and $\diam(\Omega_k) \geq \delta$, there exists a line segment $I_k$ with endpoints on $\partial \Omega_k$ that has diameter at least $\delta$. We partition this segment in three segments of equal length and denote the middle one by $J_k$. By projecting $J_k$ to $\{v_0\}^\perp$ and $\{w_0\}^\perp$ we see that there exists a constant $c>0$ depending only on the angle between $v_0 $ and $w_0$ such that 
\begin{align*}
m_1(\{ v\in \{v_0\}^\perp : L_v \cap J_k\neq \emptyset\}) +m_1(\{ w\in \{w_0\}^\perp:L_w \cap J_k\neq \emptyset\})\geq c\delta.
\end{align*}
Without loss of generality, we may assume that for infinitely many $k\in Z$ we have
\begin{align*}
m_1(\{ v\in \{v_0\}^\perp:L_v \cap J_k\neq \emptyset\}) \geq \frac{c}{2}\delta.
\end{align*}
We also shrink $Z$ so that this is the case for all $k\in Z$. Hence, if $L_v\cap J_k\neq \emptyset$ for some $k\in Z$, then the projection of $J_k$ to $\{v_0\}^\perp$ has length at least $c\delta/2$. This is also the case for the other two subsegments of $I_k\supset J_k$. Thus, $L_{v'}\cap I_k \neq \emptyset$ for all $v'\in \{v_0\}^\perp$ with $|v'-v|<c\delta/2$.

Since $\Omega_k$ is connected, if a line $L$ intersects the interior of the segment $I_k$, then it has to intersect $\Omega_k$. Setting $\varepsilon \coloneqq c\delta/2$, we conclude that for each $k\in Z$ there exists an interval $A_k\coloneqq \{ v\in \{v_0\}^\perp:L_v \cap J_k\neq \emptyset\} \subset \{v_0\}^\perp$ with $m_1(A_k)\geq \varepsilon$ such that for all $v\in A_k$ we have $L_v\cap \Omega_k\neq \emptyset$, and moreover, $L_{v'}\cap \Omega_k\neq \emptyset$ for all $v'\in \{v_0\}^\perp$ with $|v'-v|<\varepsilon$; see Figure \ref{fig:Projection}.
 
Since $K$ is compact, it follows that $\{v\in \{v_0\}^\perp: L_v\cap K\neq \emptyset\}$ is contained in a compact interval. Hence, by the Borel-Cantelli lemma we obtain that the set $\{v\in \{v_0\}^{\perp}: v\in A_k \textrm{ for infinitely many } k\in Z\}$ has positive measure. In particular, we can find a non-exceptional line $L\coloneqq L_v$ with the detour property as in Definition \ref{Definition Detour set} such that $v\in A_k$ for infinitely many $k\in Z$. By shrinking $Z$ we assume that this holds for all $k\in Z$, and that $\Omega_k$ is bounded for all $k\in Z$.

\begin{figure}
	\centering
	\begin{tikzpicture}[
    scale=5,
    axis/.style={thick, ->, >=stealth'},
    important line/.style={thick},
    dashed line/.style={dashed, thin},
    pile/.style={thick, ->, >=stealth', shorten <=2pt, shorten
    >=2pt},
    every node/.style={color=black}
    ]


		
		\begin{scope}[scale=0.7]
			
			\draw[axis] (-1,-0.2)--(0.8,1.5) node[anchor=south east] {$w_0$};
			\draw[axis] (-1,0)--(1,0) node[anchor= north west] {$v_0$};
			\draw[axis] (-1,0)--(-1,1.7) node[anchor=south east] {$\{v_0\}^\perp  $};			
			
			\begin{scope}[rotate = 60, scale=.2]
				\def\a{2.4};\def\A{1};
				\def\b{0.05};\def\B{2};
				\def\c{0.8};\def\C{1};
				\def\d{0.1};\def\D{34};
				\def\e{3};\def\f{0};
            	 	\draw [fill=black!5,smooth, domain=0:360] plot 
            			(
            			{\a*cos( \A*\x )+\b*sin(\B*\x)+\e}, 
            			{\c*sin( \C*\x ) + \d*cos(\D*\x )-\f}
            			) ;

				\def\x{0};            	
            	\node[draw, shape=circle,fill=black, scale=0.3] (a) at (
            			{\a*cos( \A*\x )+\b*sin(\B*\x)+\e}, 
            			{\c*sin( \C*\x ) + \d*cos(\D*\x )-\f}
            			) {};
            	\def\x{180};
            	\node[draw, shape=circle,fill=black, scale=0.3](b) at (
            			{\a*cos( \A*\x )+\b*sin(\B*\x)+\e}, 
            			{\c*sin( \C*\x ) + \d*cos(\D*\x )-\f}
            			) {};
            	
            	\draw[thick] (a.center)--(b.center);
            	
            	
				\pgfmathparse{\customticknum-1}
    				\foreach \i in {1,...,\pgfmathresult}
    					{
       					\draw let
       						 \p{ab}=($(b)-(a)$),
       						 \n{ab}={veclen(\x{ab}, \y{ab})}  
       						 in 
       						 ($(a)!\i*(\n{ab})/(\customticknum)!(b)!0.5*\customticklen!(b)$) -- ($(a)!\i*(\n{ab})/(\customticknum)!(b)!0.5*\customticklen!(b)$) 
       						 node[draw, shape=circle,fill=black, scale=0.3] (c\i) {};
						}  
						\node[xshift=0.35cm, yshift=0.2cm] (J) at (c2) {$J_k$};

			\end{scope}
			
			\draw  	[dashed]   (c1.center)--($(-1,0)!(c1)!(-1,1)$) node[draw, shape=circle,fill=black, scale=0.3](d1) {};			
			\draw	[dashed]   (c2.center) --($(-1,0)!(c2)!(-1,1)$)node[draw, shape=circle,fill=black, scale=0.3](d2) {};

			\node[]	(O) at ($(J)+(0.35,0.3)$) {$\Omega_k$};
			\node[] (A) at ($0.5*(d1) +0.5*(d2)+(0.07,0)$) {$A_k$};
			
		\end{scope}
	
	\end{tikzpicture}
	\caption{Projecting $J_k$ to $\{v_0\}^\perp.$} \label{fig:Projection}
\end{figure}
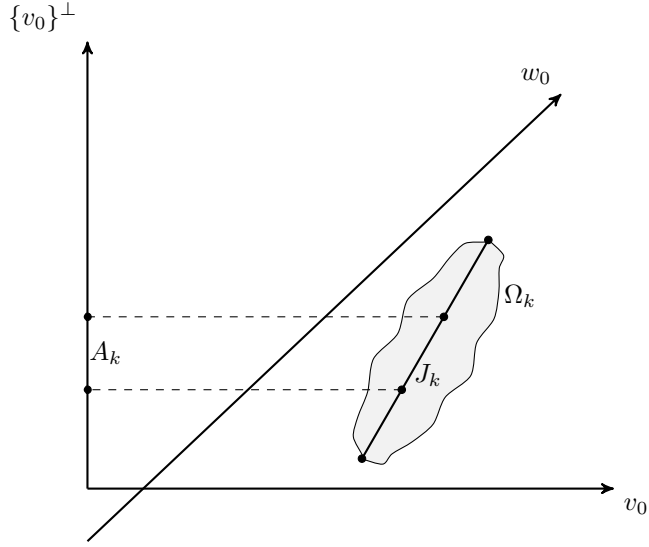

Now, observe that, by construction, both boundary lines of the $\frac{\varepsilon}{2}$-neighborhood $U$ of $L$ intersect $\Omega_k$, for all $k\in Z$. This implies that $U\setminus \br \Omega_k$ has two unbounded components, and any curve $\gamma\subset U$ connecting them must intersect $\Omega_k$. Thus,  if $\gamma$ is a detour path very close to the line $L$, it has to intersect $\Omega_k$ for all $k\in Z$, and in particular it has to intersect infinitely many sets $\Omega_k$. Thus, there is no finite set $\Gamma$ such that $\gamma\subset \bigcup_{k\in \Gamma} \br\Omega_k$. This contradicts the detour property of $L$. The proof of (ii) is completed.

In our setting, local connectedness follows immediately from (ii); see \cite[Theorem 4.4, pp.~112--113]{Wh} and also \cite[Lemma 19.5]{Mil}.

Next, we prove (iv). We fix a complementary Jordan region $\Omega_{k_0}$, and we claim that there exists a component $\Omega_l$, $l\neq k_0$, such that $\br \Omega_{k_0}\cap \br \Omega_l\neq \emptyset$. We assume that for the moment, and fix a point $z\in \br \Omega_{k_0}\cap \br \Omega_l$. We now show that this is a local cut point of $K$. Consider a small connected neighborhood $W$ of $z$ in $K$, contained in $B(z,\rho)$, where $\rho<\min \{\diam( \Omega_{k_0}), \diam( \Omega_l)\}$. This exists by the local connectedness of $K$, from (iii). Let $\beta_{1}\subset \Omega_{k_0}$ be a simple open path from a point $x_1 \in \Omega_{k_0}\setminus B(z,\rho)$ to $z$. Also, let $\beta_2 \subset \Omega_l$ be a simple open path from a point $x_2 \in  \Omega_l \setminus B(z,\rho)$ to $z$. We concatenate the  paths $\br \beta_1$ and $\br \beta_2$ with a simple path $\beta_3 \subset \R^2 \setminus (\br \beta_1\cup \br\beta_2\cup B(z,\rho))$ that connects $x_1$ and $x_2$, to obtain a loop $\beta$. Note that $\beta\cap W= \{z\}$, and  $W\setminus \beta$ is disconnected, so $z$ is a local cut point. Indeed, arbitrarily close to $z$ there are points of $\partial \Omega_{k_0}\cap W$ lying in the ``interior", and points lying in the ``exterior" of the loop $\beta$.

Since $\inter(K)=\emptyset$ by Lemma \ref{Detour empty interior}, it follows that the complementary components $\Omega_k$ are dense in $K$. Thus, a point $x\in K$ either has arbitrarily small neighborhoods intersecting infinitely many components $\Omega_k$, or every small neighborhood of $x$ intersects finitely many sets $\Omega_k$. In the first case, by (ii) we can find arbitrarily small sets $\Omega_k$ near $x$, and thus we can find local cut points near $x$, by the previous paragraph. In the latter case, again because $\inter(K)=\emptyset$, we must have $x\in \br \Omega_{k_1}\cap \br \Omega_{k_2}$ for some $k_1\neq k_2$. The argument in the previous paragraph shows that $x$ is a local cut point. Hence, indeed there exists a dense set of local cut points.

Finally, we show our initial claim. We fix a set $\Omega_{k_0}$, and we shall show that there exists $l\neq k_0$ such that $\br \Omega_{k_0}\cap \br \Omega_l\neq \emptyset$.  As in the proof of (ii), we consider the directions $v_0,w_0\in \R^2$ along which the lines with the detour property exist. Since $\partial \Omega_{k_0}$ is a Jordan curve, its projection to $\{v_0\}^\perp$ is a non-trivial interval $A$.

Consider two lines $L_{v_1}, L_{v_2}$, for $v_1,v_2\in A$, using the notation from (ii). Then the strip $U$ defined by the two lines is separated by $\partial \Omega_{k_0}$, i.e., $U\setminus \partial \Omega_{k_0}$ is disconnected. We fix a non-exceptional line $L_v$, parallel to $v_0$, between $L_{v_1}$ and $L_{v_2}$ that has the detour property. We also consider a detour path $\gamma$, as in Definition \ref{Definition Detour set} that lies in $U$. Then $\gamma$ has to intersect $\partial \Omega_{k_0}$. Note that, by definition, $\gamma$ is contained in a \textit{finite} union $\br \Omega_{k_1}\cup \dots \cup \br \Omega_{k_m}$, with $k_0\in \{k_1,\dots,k_m\}$. By the connectedness of $\gamma$, the set $\br \Omega_{k_0}$ cannot have positive distance from the other sets in the union, provided these are more than one; thus, there exists $l\in \{k_1,\dots,k_m\}\setminus \{k_0\}$ such that $\br \Omega_{k_0}\cap \br \Omega_l \neq \emptyset$. 

It remains to prove that $\{k_1,\dots,k_m\} \supsetneq \{k_0\}$. For this, it suffices to have that $\gamma$ intersects both complementary components of $\partial \Omega_{k_0}$. If $\Omega_{k_0}$ is a bounded component, then the statement is trivial, since $\gamma$ already meets the unbounded component of $K$. If $\Omega_{k_0}$ is the unbounded component, and $V$ is the (bounded) Jordan region enclosed by $\partial \Omega_{k_0}$, then $\gamma$ has to intersect $V$ as one can see by a connectedness argument. 
\end{proof}

\begin{remark}
The proof of (iv) shows the stronger conclusion that for every complementary Jordan region $\Omega_k$ there exists $l\neq k$ such that $\br \Omega_k \cap \br \Omega_l\neq \emptyset$.  In fact, by refining the argument, one can show that for every arc $J\subset \partial \Omega_k$ there exists $l\neq k$ such that $J\cap \br \Omega_l\neq \emptyset$. In particular, the set of local cut points of $K$ which lie on $\partial \Omega_k$ is dense in $\partial \Omega_k$.
\end{remark}

\begin{remark}
A set $K\subset \R^2$ satisfying the conditions in Proposition \ref{Detour dimension 2} is \textit{not} necessarily a detour set. An example is constructed as follows. Let $A_k$ be the strip $[2^{-k},2^{-k+1}]\times [0,1]$, with $2^{k}$ disjoint open squares of sidelength $2^{-k}$ removed. Then take $K\coloneqq \br{\bigcup_{k=1}^\infty A_k}$. It is immediate that $K$ has the properties of Proposition \ref{Detour dimension 2}. However, $\{0\}\times [0,1] \subset K$, and this prevents detours of lines intersecting $\{0\} \times [0,1]$, since this set intersects the closure of only one complementary component of $K$. For constructions which are ``partial converses" to the preceding proposition see Section \ref{Section Examples}.
\end{remark}

\begin{prop}\label{Detour Zero measure}
Let $K\subset \R^n$ be a detour set. If $m_n(\partial \Omega_k)=0$ for all $k\in \N\cup \{0\}$, and
\begin{align*}
\sum_{k= 1}^\infty \diam(\Omega_k)^n <\infty,
\end{align*}
then $m_n(K)=0$.
\end{prop}
Recall that $\Omega_0$ is the unbounded complementary component of $K$, and $\{\Omega_k\}_{k\geq 1}$ is the family of the bounded components.

\begin{remark}
This proposition already shows the special feature of a detour set. To illustrate that, we note that each square Sierpi\'nki carpet in the plane trivially satisfies the assumptions that the boundaries of complementary components, which are squares, have measure zero, and that the sum of the areas of the squares is finite. However, by modifying the construction of the standard Sierpi\'nski carpet, it is very easy to construct carpets having positive area. The reason is, of course, that carpets fail to be detour sets, because the closures of their complementary components are disjoint. 
\end{remark}

\begin{proof}
We fix one of the directions $v_0$ as in Definition \ref{Definition Detour set}, along which the lines with the detour property exist. We write $L_v=L_{v_0}+v$, $v\in \{v_0\}^\perp\simeq \R^{n-1}$, to denote the lines parallel to $v_0$. The assumption that  $m_n(\partial \Omega_k)=0$ and Fubini's theorem imply that
\begin{align}\label{Detour Zero measure m_n(K)}
m_n(K)=m_n(K\setminus \bigcup_{k=0}^\infty \partial \Omega_k) = \int m_1(L_v\cap K\setminus \bigcup_{k=0}^\infty \partial \Omega_k) \, dm_{n-1}(v),
\end{align}
hence, we need to obtain an estimate for 
\begin{align*}
m_1(L_v\cap K\setminus \bigcup_{k=0}^\infty \partial \Omega_k)=m_1(L_v\setminus \bigcup_{k=0}^\infty \br \Omega_k).
\end{align*}

We fix a line $L=L_v$, and a finite set $Z\subset \N\cup \{0\}$ that contains $0$. The set $L\setminus \bigcup_{k\in Z} \br\Omega_k$ is an open subset of the line $L$, so it is the union of at most countably many open intervals. We enumerate the intervals that intersect $K$ by $\{I_j\}_{j\in \N}$. In case these are finitely many, we set $I_j=\emptyset$ for large $j$. Note that all the intervals $I_j$ are bounded (since $0\in Z$) and have their endpoints on sets $\partial \Omega_k$, $k\in Z$, but otherwise they are disjoint from $\bigcup_{k\in Z}\br \Omega_k$.

We now fix one interval $I=I_j \eqqcolon(x,y)$ and estimate $\diam(I)$. Let $\delta>0$, and consider a compact subinterval $[a,b]\subset (x,y)$ such that $|x-a|<\delta$ and $|y-b|<\delta$. Let $\gamma$ be a detour path as in Definition \ref{Definition Detour set} that lies $\eta$-close to the line $L$, where $0<\eta<\delta$ is so small that  the $\eta$-neighborhood of $[a,b]$ does not intersect the closed set $\bigcup_{k\in Z}\br \Omega_k$. Then we can choose a subpath $\gamma_I\subset \gamma$ that is entirely contained in the $\eta$-neighborhood of $[a,b]$, and whose endpoints $c,d$ satisfy $|a-c|<\eta<\delta$ and $|b-d|<\eta<\delta$; see Figure \ref{fig:detour}. In particular, $\gamma_I$ does not intersect $\bigcup_{k\in Z} \br \Omega_k$ and stays entirely in $\bigcup_{k\notin Z} \br \Omega_k$, and in fact, it is contained in the union of only finitely many sets $\br\Omega_k$, $k\in \Gamma(I)$; here $\Gamma(I)=\{k\in \Gamma : \gamma_I \cap \br\Omega_k\neq \emptyset\}$ and $\Gamma$ is as in Definition \ref{Definition Detour set}(ii), so that $\gamma\subset \bigcup_{k\in \Gamma} \br \Omega_k$. Hence,
\begin{align}\label{Detour Zero measure One Estimate}
\diam(I)\leq |x-c|+\diam(\gamma_I) + |d-y| \leq 4\delta +\sum_{k\in \Gamma(I)} \diam(\Omega_k).
\end{align}

We now fix finitely many intervals $I_1,\dots,I_N$ and perform the same procedure for each interval $I_j=(x_j,y_j)$ to obtain paths $\gamma_j$ as above, that do not intersect $\bigcup_{k\in Z} \br\Omega_k$ and are contained in the union of finitely many sets $\br\Omega_k$, $k\in \Gamma(I_j)$. We note that $\delta$ is still fixed. By running a greedy algorithm we group the indices $j$ into finite sets $A_i$, $i=1,\dots,M$, such that:
\begin{enumerate}[(i)]
\item intervals $I_j$ with indices $j$ lying in distinct sets $A_i$ intersect disjoint sets of regions $\br\Omega_k$, $k\in \bigcup_{j=1}^N \Gamma(I_j)$, and
\item $\bigcup_{j\in A_i} \bigcup_{k\in \Gamma(I_j)} \br\Omega_k$ is connected for all $i=1,\dots,M$.
\end{enumerate}

We fix $i$ and points $z,w \in \bigcup_{j\in A_i} I_j$. Assume that $z\in I_{j_1}$ and $w \in I_{j_2}$. We have
\begin{align*}
|z-w|\leq \diam(I_{j_1}) + |x_{j_1}-x_{j_2}|+\diam(I_{j_2}).
\end{align*}
We bound the first and the last term using \eqref{Detour Zero measure One Estimate}. For the middle term we follow the procedure used in the proof of \eqref{Detour Zero measure One Estimate}. Namely, we have 
\begin{align*}
|x_{j_1}-x_{j_2}|&\leq \dist(x_{j_1},\gamma_{j_1}) + \diam(\bigcup_{j\in A_i} \bigcup_{k\in \Gamma(I_j)} \br\Omega_k) +\dist(x_{j_2},\gamma_{j_2})\\
&\leq 4\delta + \sum_{{k\in \bigcup_{j\in A_i}\Gamma(I_j)}} \diam(\Omega_k),
\end{align*}
where in the last step we used the connectedness of $\bigcup_{j\in A_i} \bigcup_{k\in \Gamma(I_j)} \br\Omega_k$, by property (ii) of the set $A_i$. Hence, we obtain the estimate
\begin{align*}
|z-w|&\leq 12\delta +  \sum_{k\in \Gamma(I_{j_1})} \diam(\Omega_k)+\sum_{k\in \Gamma(I_{j_2})} \diam(\Omega_k)\\
&\qquad\qquad + \sum_{{k\in \bigcup_{j\in A_i}\Gamma(I_j)}} \diam(\Omega_k).
\end{align*}
This yields the bound
\begin{align*}
\diam( \bigcup_{j\in A_i} I_j) \leq 12\delta+ 3\cdot \sum_{{k\in \bigcup_{j\in A_i}\Gamma(I_j)}} \diam(\Omega_k).
\end{align*}
Hence, 
\begin{align*}
m_1( \bigcup_{j=1}^N I_j) &= m_1( \bigcup_{i=1}^M \bigcup_{j\in A_i}I_j)\leq  \sum_{i=1}^M \diam(\bigcup_{j\in A_i} I_j) \\
&\leq 12M\delta + 3 \sum_{i=1}^M \sum_{{k\in \bigcup_{j\in A_i}\Gamma(I_j)}} \diam(\Omega_k).
\end{align*}
Recall that the detour paths $\gamma_j$ do not intersect $\br\Omega_k$, $k\in Z$, and 
\begin{align*}
\bigcup_{j=1}^N \Gamma(I_j)\subset \Gamma\subset  \{k:\br\Omega_k\cap L\neq \emptyset\}\setminus Z,
\end{align*}
by definition. Using property (i) of the sets $A_i$ we can combine the two sums in one sum:
\begin{align*}
m_1( \bigcup_{j=1}^N I_j) \leq 12M\delta +3\sum_{\substack{k:\br\Omega_k\cap L\neq \emptyset\\k\notin Z}} \diam(\Omega_k).
\end{align*}
Letting $\delta \to 0$ we obtain 
\begin{align*}
m_1( \bigcup_{j=1}^N I_j)  \leq 3\sum_{\substack{k:\br\Omega_k\cap L\neq \emptyset\\k\notin Z}} \diam(\Omega_k).
\end{align*}
Since there is no dependence on $N$ in the right-hand side, we can let $N\to\infty$, so we have
\begin{align*}
m_1(L\setminus \bigcup_{k=0}^\infty \br\Omega_k )&\leq m_1(L\setminus \bigcup_{k\in Z} \br\Omega_k)= m_1( \bigcup_{j=1}^\infty I_j)\leq  3 \sum_{\substack{k:\br\Omega_k\cap L\neq \emptyset\\k\notin Z}} \diam(\Omega_k).
\end{align*}

Finally, we integrate over all lines $L=L_v=L_{v_0}+v$, $v\in \{v_0\}^\perp\simeq \R^{n-1}$, and using Fubini's theorem we have
\begin{align*}
\int m_1(L_v\setminus \bigcup_{k=0}^\infty \br\Omega_k ) \,dm_{n-1}(v) &\leq 3 \sum_{k\notin Z} \diam(\Omega_k)\int_{\br\Omega_k\cap L_v\neq \emptyset} dm_{n-1}(v)\\
&\leq 3\sum_{k\notin Z} \diam(\Omega_k)^n.
\end{align*}
The latter series is convergent, by assumption. Hence, if we let $Z\to \N\cup \{0\}$ we obtain $m_n(K)=0$ by \eqref{Detour Zero measure m_n(K)}, as desired.
\end{proof}

\section{H\"older domains}\label{Section Holder}
In this section we discuss H\"older domains and we prove some lemmas that will be needed in the proof of the Main Theorem. We pay particular attention to the implicit constants in the various inequalities proved, and their dependence on the data.

Let $D\subset \R^n$ be a domain, i.e., a connected open set. For a point $x\in D$ let $\delta_D(x)\coloneqq \dist(x,\partial D)$. We define the \textit{quasihyperbolic distance} of two points $x_1,x_2\in D$ to be
\begin{align*}
k_D(x_1,x_2)= \inf_\gamma \int_\gamma \frac{1}{\delta_D(x)} ds,
\end{align*}
where the infimum is taken over all rectifiable paths $\gamma \subset D$ connecting $x_1$ and $x_2$.

\begin{definition}\label{Definition Holder}
A domain $D\subset \R^n$ is an $(\alpha,c)$-H\"older domain if there exists a \textit{basepoint} $x_0\in D$ and constants $\alpha \in (0,1]$, $c>0$ such that 
\begin{align}\label{Definition Holder inequality}
k_D(x,x_0) \leq \frac{1}{\alpha} \log \left( \frac{\delta_D(x_0)}{\delta_D(x)}\right) +c
\end{align}
for all $x\in D$. 
\end{definition}
We often drop the notation $(\alpha,c)$ and we say that $D$ is a H\"older domain if the constants are irrelevant.

A curve $\gamma \subset D$ is called a \textit{quasihyperbolic geodesic} if for any two points $x_1,x_2\in \gamma$ we have
\begin{align*}
k_D(x_1,x_2)=\int_{\gamma|_{[x_1,x_2]}} \frac{1}{\delta_D(x)} ds,
\end{align*}
where $\gamma|_{[x_1,x_2]}$ denotes the subpath of $\gamma$ between $x_1$ and $x_2$. A compactness argument shows that for any two points $x_1,x_2 \in D$ there exists a {quasihyperbolic geodesic} $\gamma \subset D$ that connects them; see \cite[Theorem 2.5.14]{BBI} for an argument.

For a domain $D\subset \R^n$ we can consider its \textit{Whitney cube decomposition} $W(D)$, which is a collection of closed dyadic cubes $Q\subset D$, with the following properties:
\begin{enumerate}[(i)]
\item the cubes of $W(D)$ have disjoint interiors, and  $\bigcup_{Q\in W(D)} Q=D$,
\item $\ell(Q)\simeq \delta_D(x)$ for all $x\in Q$, $Q\in W(D)$, and
\item if $Q_1\cap Q_2\neq \emptyset$, $Q_1,Q_2\in W(D)$, then $\frac{1}{4}\leq \frac{\ell(Q_1)}{\ell(Q_2)}\leq 4$.
\end{enumerate}
Here $\ell(Q)$ denotes the sidelength of the cube $Q$. By refining each cube $Q\in W(D)$, we may also assume that $k_D(x_1,x_2)\leq \frac{1}{3}$ for all $x_1,x_2\in Q$. For $j\in \N$ we define
\begin{align*}
D_j= \{Q\in W(D): k_D(x_0,Q)\leq j\},
\end{align*}
and $D_0\coloneqq \emptyset$. Each Whitney cube $Q$ is contained in $D_j\setminus D_{j-1}$ for some $j\in \N$. Also, we denote by $x_Q$ the center of the cube $Q$. We say that two Whitney cubes $Q_1$ and $Q_2$ with, say, $\ell(Q_1)\geq \ell(Q_2)$ are \textit{adjacent}, if a face of $Q_2$ is contained in a face of $Q_1$. In other words, the intersection $Q_1\cap Q_2$ contains an open subset of a hyperplane $\R^{n-1}\subset \R^n$. We also allow the possibility that $Q_1=Q_2$.

\begin{lemma}\label{Holder length estimate}
Let $D\subset \R^n$ be an $(\alpha,c)$-H\"older domain with basepoint $x_0$, and fix $\beta>0$. For each quasihyperbolic geodesic $\gamma$ starting at $x_0$ we have
\begin{align*}
\sum_{\substack{Q\in W(D)\\ Q\cap \gamma\neq \emptyset}} \ell(Q)^{\beta} \lesssim \delta_D(x_0)^{\beta}.
\end{align*}
The implicit constants depend only on $\alpha,c,n,\beta$ and not on $\gamma$.
\end{lemma}
\begin{proof}
By \eqref{Definition Holder inequality} we have
\begin{align*}
\delta_D(x) \leq \delta_D(x_0) e^{\alpha c} \exp(-\alpha k(x,x_0)).
\end{align*}
If $Q\in D_j\setminus D_{j-1}$, then $k(x,x_0)\geq j-1$ for all $x\in Q$, thus
\begin{align*}
\delta_D(x) \lesssim \delta_D(x_0) \exp(-\alpha j)
\end{align*}
with implicit constant depending on $\alpha,c$. 

Since $\gamma$ is a quasihyperbolic geodesic starting at $x_0$, there exists a uniform constant $N\in \N$ such that for each $j\in \N$ there exist at most $N$ Whitney cubes $Q\in D_j\setminus D_{j-1}$ intersecting $\gamma$. This follows from the observation that the quasihyperbolic distance of two points not contained in adjacent cubes is comparable to the number of Whitney cubes that the geodesic intersects, between the two points. 

We now have
\begin{align*}
 \sum_{\substack{Q\in W(D)\\ Q\cap \gamma\neq \emptyset}} \delta_D(x_Q)^{\beta}&= \sum_{j=1}^\infty  \sum_{\substack{Q\in D_j\setminus D_{j-1}\\ Q\cap \gamma\neq \emptyset}} \delta_D(x_Q)^{\beta}\lesssim N\delta_D(x_0)^{\beta} \sum_{j=1}^\infty  \exp(-\alpha\beta j)\lesssim \delta_D(x_0)^\beta.
\end{align*} 
By the properties of the Whitney cubes we have $\delta_D(x_Q)\simeq \ell(Q)$ for all $Q\in W(D)$, so the conclusion follows.
\end{proof}

\begin{corollary}\label{Holder is Quasiball}
Let $D\subset \R^n$ be an $(\alpha,c)$-H\"older domain with basepoint $x_0$. For any quasihyperbolic geodesic $\gamma\subset D$ starting at $x_0$ we have
\begin{align*}
\length(\gamma)\lesssim \delta_D(x_0).
\end{align*}
In particular, 
\begin{align*}
\diam(D)&\simeq \delta_D(x_0), \quad \textrm{and}\\
 m_n(D)&\simeq \delta_D(x_0)^n\simeq \diam(D)^n,
\end{align*}
with implicit constants depending only on $\alpha,c,n$.
\end{corollary}
\begin{proof}
We fix a quasihyperbolic geodesic $\gamma$ as in the statement. For any Whitney cube $Q\in W(D)$ and points $x_1,x_2\in Q$ we have $k_D(x_1,x_2)\lesssim 1$. Thus, if we denote by $\gamma_Q$ the portion of the path $\gamma$ inside $Q$, then
\begin{align*}
\frac{\length(\gamma_Q)}{\ell(Q)} \lesssim \int_{  \gamma_Q}\frac{1}{\delta_D(x)} ds \lesssim 1.
\end{align*}
Hence,
\begin{align*}
\length(\gamma) \leq \sum_{\substack{Q\in W(D)\\ Q\cap \gamma\neq \emptyset}} \length(\gamma_Q)\lesssim \sum_{\substack{Q\in W(D)\\ Q\cap \gamma\neq \emptyset}} \ell(Q).
\end{align*}
Applying Lemma \ref{Holder length estimate} for $\beta=1$ we obtain the desired bound $\delta_D(x_0)$.

The other parts follow trivially since $B(x_0,\delta_D(x_0))\subset D \subset B(x_0,\diam(D))$.
\end{proof}

\begin{lemma}\label{Geodesics to boundary}
Let $D\subset \R^n$ be a H\"older domain with basepoint $x_0$. Then each point $x\in \partial D$ is the landing point of a quasihyperbolic geodesic $\gamma$ passing through $x_0$.
\end{lemma}

This result is proved in \cite[Theorem 3.21]{MV} for \textit{uniform domains}, a notion strictly stronger than that of H\"older domains.

\begin{proof}
Let $x\in \partial D$ and consider a sequence $x_k\in D$ with $x_k \to x$. Let $\gamma_k$ be a quasihyperbolic geodesic from $x_k$ to  $x_0$. By Corollary \ref{Holder is Quasiball} we have that $\length(\gamma_k)$ is uniformly bounded, independently of $k$. Applying the Arzel\`a-Ascoli theorem (after reparametrizing the paths $\gamma_k$) we obtain a subsequence, still denoted by $\gamma_k$, that converges uniformly to a rectifiable  path $\gamma \subset \br{D}$, connecting $x$ to $x_0$, with $\length(\gamma)\leq \liminf_{k\to\infty} \length (\gamma_k)<\infty$; see \cite[Theorem 2.5.14]{BBI} for a detailed argument. We consider the rescaled (Euclidean) arc-length parametrization $\gamma\colon [0,1]\to \br D$, with $\gamma(0)=x_0$ and $\gamma(1)=x$. We now wish to show that $\gamma$, or rather, its open subpath restricted to $[0,1)$, is a quasihyperbolic geodesic. One first has to argue that $\gamma|_{[0,1)}$ is contained in $D$. 

Assuming this, one can easily derive that $\gamma$ has to be a geodesic. It is a general fact that if a sequence of geodesics $\zeta_k$ in a \textit{length space} $(X,d)$ converges uniformly to a path $\zeta$ in $X$, then $\zeta$ is also a geodesic; see \cite[Theorem 2.5.17]{BBI}. In our case one needs to apply this principle to all compact subpaths $\zeta\subset \subset D$ of the open path $\gamma$, and suitable subpaths $\zeta_k$ of $\gamma_k$ converging to $\zeta$.

Finally, we argue that $\gamma|_{[0,1)}\subset D$. If there is some time $t_1\in (0,1)$ with $\gamma(t_1) \in \partial D$, we consider the first such time. Note that the curve $\gamma$ cannot be constant on $(t_1,1)$ since it is parametrized by arc-length. Hence, either there exists $t_2\in (t_1,1)$ with $\gamma(t_2)\in D$, or $\{\gamma(t): t\in [t_1,1]\}$ is a non-trivial continuum contained in $\partial D$. The first scenario can be easily excluded, because all the quasihyperbolic geodesics connecting $x_0$ to points in a small neighborhood of $\gamma(t_2)$ must remain in a fixed compact subset of $D$. Thus, the limiting path $\gamma|_{[0,t_2]}$ is also contained in the same compact set, and cannot meet $\partial D$, a contradiction.

In the second scenario, there exists a point $y=\gamma(t_3)\in \partial D$, $t_3\in (t_1,1)$, with $y\neq x$. Now we use a technical lemma from \cite{SS}.

\begin{lemma}[{\cite[Theorem 3]{SS}}]
Let $D\subset \R^n$ be a H\"older domain with basepoint $x_0$. Then there exist constants $c_1,c_2>0$ such that whenever $\zeta$ is a quasihyperbolic geodesic joining $x_0$ to a point $z\in D$ we have
\begin{align*}
k_D(y,x_0) \leq c_1 \log \left(\frac{\delta_D(x_0)}{\length(\zeta|_{[y,z]})} \right) +c_2
\end{align*}
for all $y\in \zeta$.
\end{lemma}
We apply this for the paths $\gamma_k$ connecting $x_0$ to $x_k$, and for points $y_k\in \gamma_k$ which are very close to $y\in \partial D$. Since $x_k\to x$ and $y_k\to y$, but $y\neq x$, there exists $\varepsilon>0$ such that $|x_k-y_k|\geq \varepsilon$ for all sufficiently large $k$. Therefore,
\begin{align*}
k_D(y_k,x_0)&\leq c_1 \log \left(\frac{\delta_D(x_0)}{\length(\gamma_k|_{[y_k,x_k]})} \right) +c_2\\
&\leq c_1 \log \left(\frac{\delta_D(x_0)}{\varepsilon} \right)+c_2
\end{align*}
for all large $k$. Note that $y_k\to \partial D$, thus $k_D(y_k,x_0)\to \infty$ as $k\to \infty$, and this leads to a contradiction.
\end{proof}

\begin{definition}\label{Defintion Shadow}
Let $D\subset \R^n$ be a H\"older domain with basepoint $x_0$. For a Whitney cube $Q \in W(D)$ we define the \textit{shadow} of $Q$ to be the set $SH(Q)$ of points  $x\in \partial D$, such that there exists a quasihyperbolic geodesic, intersecting $Q$, from the basepoint $x_0$ to the point $x$. We also define
\begin{align*}
s(Q)\coloneqq  \diam( SH(Q)).
\end{align*}
\end{definition}

\begin{lemma}\label{Shadow-QH distance}
Let $D\subset \R^n$ be an $(\alpha,c)$-H\"older domain with basepoint $x_0$. We have
\begin{align*}
\sum_{Q\in W(D)} s(Q)^n \lesssim \int_D k_D(x,x_0)^n dx<\infty.
\end{align*}
\end{lemma}

The first inequality is proved in \cite[Section 3, p. 275]{JS} and, in fact, the implicit constant depends only on $n$. The integrability of the quasihyperbolic distance in H\"older domains is the conclusion of \cite[Theorem 2]{SS}.

We also include a technical lemma.

\begin{lemma}\label{Holder Shadows Connections}
Let $D\subset \R^n$ be a H\"older domain with basepoint $x_0$. For every $\varepsilon>0$ and every $x\in \partial D$ there exists $r>0$ such that for all points $y\in B(x,r)\cap \partial D$ there exist adjacent Whitney cubes $Q_x,Q_y\in W(D)$ with $x\in SH(Q_x)$, $y\in SH(Q_y)$, and $\ell(Q_x)\leq \varepsilon$, $\ell(Q_y)\leq \varepsilon$. Furthermore, the quasihyperbolic geodesic from $x$ to $x_0$, restricted to a subpath from $x$ to $Q_x$, does not intersect any Whitney cube $Q$ with $\ell(Q)>\varepsilon$. The same holds for the corresponding subpath of the geodesic from $y$ to $x_0$.
\end{lemma}
\begin{proof}
Assume there exist $\varepsilon >0$, $x\in \partial D$ and a sequence $x_n\to x$ such that the conclusion fails. Then the quasihyperbolic geodesics $\gamma_n$ from $x_n$ to $x_0$ subconverge uniformly to a quasihyperbolic geodesic $\gamma$ from $x$ to $x_0$, as in the proof of Lemma \ref{Geodesics to boundary}. Let $Q_x$ be the last Whitney cube of sidelength smaller than $\varepsilon/4$ that $\gamma$ intersects, as it travels from $x$ to $x_0$. For a fixed sufficiently large $n$, $\gamma_n$ intersects an adjacent cube $Q_y$ of $Q_x$, by the uniform convergence. By the properties of Whitney cubes we have that $\ell(Q_y)\leq 4\ell(Q_x) <\varepsilon$. Again by uniform convergence we see that if $\gamma_n$ intersects a cube $Q_y'$ with $\ell(Q_y')\geq \varepsilon$ before hitting $Q_y$ (there are finitely many such cubes $Q_y'$ by the boundedness of $D$), then $\gamma$ intersects an adjacent cube $Q_x'$ with $\ell(Q_x')\geq \varepsilon/4$, before hitting $Q_x$, which is a contradiction.  The whole argument leads to a contradiction, so the statement of the lemma is true.
\end{proof}

Finally, we need the following lemma:

\begin{lemma}[{\cite[Corollary 4]{SS}}]\label{Holder measure zero}
If $D\subset \R^n$ is a H\"older domain, then $m_n(\partial D)=0$.
\end{lemma}
In fact the Hausdorff dimension of $\partial D$ is strictly less than $n$, but we will not need this result.

\section{Sobolev function estimates}\label{Section Sobolev}

For an integrable function $f\colon \R^n\to \R$ we use the notation
$$\mint{-}_Qf \coloneqq  \frac{1}{m_n(Q)}\int_Q f$$ 
for the average of $f$ over a measurable set $Q\subset \R^n$ of finite, non-zero measure.

\begin{lemma}\label{Sobolev Neighbor cubes}
Let $D\subset \R^n$ be an open set, and $f\in W^{1,p}(D)$ for some $1\leq p<\infty$. For all adjacent cubes $Q_1,Q_2\in W(D)$ we have
\begin{align*}
\left| \mint{-}_{Q_1} f - \mint{-}_{Q_2} f\right|\lesssim  \ell(Q_1)\mint{-}_{Q_1} |\nabla f| +\ell(Q_2)\mint{-}_{Q_2}|\nabla f|,
\end{align*} 
where the implicit constant depends only on $n$. 
\end{lemma}

The proof follows easily from the fundamental theorem of calculus and Fubini's theorem in case the cubes have equal sidelength. We include the proof for the sake of completeness.

\begin{proof}
By density, it suffices to assume that $f\in C^1(D)$. If $Q_1$ and $Q_2$ are cubes of equal sidelength, we assume that $Q_1=[0,h]^n$ and $Q_2=[h,2h]\times [0,h]^{n-1}$.  Then for  $z\in Q_1$ and $e_1\coloneqq (1,0,\dots,0)\in \R^n$ we have
\begin{align*}
|f(z)- f(z+he_1)| \leq \int_0^h |\nabla f(z+te_1)| \, dt.
\end{align*}
Note that $z+ he_1 \in Q_2$. We write $z=(x,y)\in \R \times \R^{n-1}$, and we have
\begin{align*}
\left| \int_{Q_1} f - \int_{Q_2} f \right|&= \left| \int_{Q_1} (f(z)-  f(z+he_1)\,dm_n(z) \right|\\
&\leq \int_{y\in [0,h]^{n-1}} \int_0^h |f(x,y)-f((x,y)+he_1)|\,dm_1(x) dm_{n-1}(y)\\
&\leq \int_{y\in [0,h]^{n-1}}\int_0^h \int_0^h|\nabla f((x,y)+te_1)| \,dt dm_1(x) dm_{n-1}(y)\\
&\leq h \int_{Q_1\cup Q_2} |\nabla f| = \ell(Q_1) \int_{Q_1}|\nabla f| + \ell(Q_2) \int_{Q_2} |\nabla f|.
\end{align*}

Now, if $\ell(Q_1)<\ell(Q_2)$ we either have $\ell(Q_2)=2\ell(Q_1)$ or $\ell(Q_2)=4\ell(Q_1)$. We treat only the first case, since the second is similar. We subdivide $Q_2$ into $2^n$ dyadic cubes $\{\widetilde Q_i\}_{i=1}^{2^n}$ of sidelength equal to $\ell(Q_1)$. Note that for each $\widetilde Q_i$ there exists a chain of distinct cubes $\Delta_0,\dots,\Delta_m$, such that for each $j\geq 1$ the cube $\Delta_j$ is equal to some $\widetilde Q_{k}$, $\Delta_0=Q_1$, $\Delta_m= \widetilde Q_i$, and  $\Delta_j$ is adjacent to $\Delta_{j+1}$. Hence, applying  the estimate for adjacent cubes of equal size $m$ times we obtain
\begin{align*}
\left| \mint{-}_{Q_1}f-\mint{-}_{\widetilde Q_i}f \right| &\leq \ell(Q_1)\mint{-}_{Q_1}|\nabla f| + 2\sum_{j=1}^m \ell(\Delta_j)\mint{-}_{\Delta_j} |\nabla f|\\
&\leq \ell(Q_1)\mint{-}_{Q_1}|\nabla f| + 2^n \ell(Q_2) \mint{-}_{Q_2} |\nabla f|.
\end{align*}
Thus,
\begin{align*}
\left| \mint{-}_{Q_1} f-\mint{-}_{Q_2}f\right|& \leq \sum_{i=1}^{2^n}\frac{1}{2^n}\left| \mint{-}_{Q_1} f-\mint{-}_{\widetilde Q_i}f\right| \\
&\leq \ell(Q_1)\mint{-}_{Q_1}|\nabla f| + 2^n \ell(Q_2) \mint{-}_{Q_2} |\nabla f|. 
\end{align*}
The claim is proved. 
\end{proof}

\begin{prop}\label{Sobolev Holder estimate}
Assume that $D\subset \R^n$ is an $(\alpha,c)$-H\"older domain, and let $f\in W^{1,p}(D)$, with $p>n$. If $f$ extends continuously to $\br D$, then for all $x,y\in \partial D$ we have
\begin{align*}
|f(x)-f(y)|\lesssim \diam(D) \left( \mint{-}_{D} |\nabla f|^p \right)^{1/p}.
\end{align*}
The implicit constant depends only on $\alpha,c,p,n$. 
\end{prop}
\begin{proof}
Let $x_0$ be the basepoint of the H\"older domain $D$. Consider  the concatenation $\gamma$ of two quasihyperbolic geodesics: one from $x_0$ to $x$, and one from $x_0$ to $y$. These exist by Lemma \ref{Geodesics to boundary}.  The Whitney cubes intersecting $\gamma$ contain  a bi-infinite chain of adjacent cubes $\{Q_i\}_{i\in \Z} \subset W(D)$, such that $Q_i\to x$ as $i\to -\infty$ and $Q_i\to y$ as $i\to\infty$. By continuity, Lemma \ref{Sobolev Neighbor cubes}, and H\"older's inequality, we have
\begin{align*}
|f(x)- f(y)|&\leq \sum_{i\in \Z} \left| \mint{-}_{Q_i} f - \mint{-}_{Q_{i-1}} f\right| \lesssim \sum_{i\in \Z} \ell(Q_i) \left(\mint{-}_{Q_i} |\nabla f|^p\right)^{1/p}.
\end{align*}
By H\"older's inequality, the latter is bounded above by 
\begin{align*}
\left(\sum_{i\in \Z} \ell(Q_i)^{(1-n/p)p'}\right)^{1/p'} \left( \sum_{i\in \Z} \int_{Q_i} |\nabla f|^p \right)^{1/p}.
\end{align*}
We bound the last sum trivially by $\left(\int_D |\nabla f|^p \right)^{1/p}$. For the first sum we observe that $\beta\coloneqq (1-n/p)p' >0$ and then using Lemma \ref{Holder length estimate} we obtain the bound $\diam(D)^{1-n/p}$. These two terms, combined, yield the result.
\end{proof}

The next proposition is proved exactly as \cite[Proposition 1]{JS}. The statement is different, though, so we include the proof here.

\begin{prop}\label{Sobolev Holder boundary}
Assume that $D\subset \R^n$ is a H\"older domain, and $f\in W^{1,p}(D)$ for some $p\geq n$. If $f$ extends  continuously to $\br D$, then 
\begin{align*}
m_1(f(L\cap \partial D)) =0
\end{align*}
for a.e.\ line $L$ parallel to a fixed direction.
\end{prop}
\begin{proof}
We fix a line $L$ parallel to a direction $v_0\in \R^n$. We fix $\varepsilon>0$, and for each $x\in L\cap \partial D$ we consider a ball $B(x,r)$ such that the conclusion of Lemma \ref{Holder Shadows Connections} holds. By the compactness of $L\cap \partial D$ we only need finitely many such balls $B(x_i,r_i)$, $i=1,\dots,N$, to cover $L\cap \partial D$. We have
\begin{align}\label{Sobolev Holder boundary H1 bound}
m_1(f(L\cap \partial D)) = m_1 ( \bigcup_{j=1}^N f(L\cap \partial D \cap B(x_i,r_i))),
\end{align}
which we will bound soon by a diameter bound.

Let $i\in \{1,\dots,N\}$ and $y_i\in L\cap \partial D \cap B(x_i,r_i)$. By Lemma \ref{Holder Shadows Connections} we can connect $y_i$ to $x_i$ by a path $\gamma_i$ that that is a concatenation of two quasihyperbolic geodesics with a segment inside two adjacent Whitney cubes, with the property that $\gamma_i$ intersects only Whitney cubes of sidelength at most $\varepsilon$. By continuity and Lemma \ref{Sobolev Neighbor cubes}, we have the estimate 
\begin{align*}
|f(x_i)-f(y_i)|\lesssim \sum_{\substack {Q\in W(D)\\ Q\cap \gamma_i \neq\emptyset}}\ell(Q) \mint{-}_Q |\nabla f|,
\end{align*}
and thus, by choosing a suitable point $y_i$, we have
\begin{align*}
\diam( f(L\cap \partial D \cap B(x_i,r_i))) \lesssim \sum_{\substack {Q\in W(D)\\ Q\cap \gamma_i \neq\emptyset}}\ell(Q) \mint{-}_Q |\nabla f|.
\end{align*}

In case there exist $i,j\in \{1,\dots,N\}$ with $i\neq j$ such that $\gamma_i$ and $\gamma_j$ intersect some common Whitney cube, we concatenate them inside the cube to obtain the better bound
\begin{align*}
\diam( f( L\cap \partial D \cap (B(x_i,r_i)\cup B(x_j,r_j)))) \lesssim \sum_{\substack {Q\in W(D)\\ Q\cap (\gamma_i\cup \gamma_j) \neq\emptyset}}\ell(Q) \mint{-}_Q |\nabla f|.
\end{align*}

By doing finitely many concatenations (e.g.\ as in the construction of the sets $A_i$ in the proof of Proposition \ref{Detour Zero measure}) we obtain a new family of paths $ \gamma_i'$, $i=1,\dots,N'$, that intersect disjoint sets of Whitney cubes, and \eqref{Sobolev Holder boundary H1 bound} implies
\begin{align*}
m_1(f(L\cap \partial D)) \lesssim \sum_{i=1}^{N'} \sum_{\substack {Q\in W(D)\\ Q\cap \gamma_i' \neq\emptyset}}\ell(Q) \mint{-}_Q |\nabla f|.
\end{align*}
Note that if $Q\cap \gamma_i'\neq \emptyset$ then $SH(Q)\cap L \neq \emptyset$ and $\ell(Q)\leq \varepsilon$ (by Lemma \ref{Holder Shadows Connections}). Since no cube is used twice in the above sums, we obtain
\begin{align*}
m_1(f(L\cap \partial D)) \lesssim \sum_{\substack{Q: SH(Q)\cap L\neq \emptyset\\ \ell(Q)\leq \varepsilon}} \ell(Q) \mint{-}_Q |\nabla f|.
\end{align*}

Now, we integrate over all lines $L=L_v\coloneqq L_{v_0} +v$, $v\in \{v_0\}^\perp\simeq \R^{n-1}$, and we have by Fubini's theorem
\begin{align*}
\int m_1(f(L_v\cap \partial D)) \, d&m_{n-1}(v)\\
& \lesssim \int \left(\sum_{\substack{Q: SH(Q)\cap L_v\neq \emptyset\\ \ell(Q)\leq \varepsilon}} \ell(Q) \mint{-}_Q |\nabla f| \right) \, dm_{n-1}(v)\\
&\simeq\sum_{Q: \ell(Q)\leq \varepsilon} \ell(Q) \left(\mint{-}_Q |\nabla f|\right) \int_{SH(Q) \cap L_v\neq \emptyset} dm_{n-1}(v)\\
&\lesssim \sum_{Q: \ell(Q)\leq \varepsilon} \ell(Q) \mint{-}_Q |\nabla f|s(Q)^{n-1}\\
&\lesssim \sum_{Q:\ell(Q)\leq \varepsilon} \ell(Q)s(Q)^{n-1} \left( \mint{-}_Q |\nabla f|^p\right)^{1/p}\\
&\lesssim \left(\int_D |\nabla f|^p\right)^{1/p} \Bigg( \sum_{Q:\ell(Q)\leq \varepsilon} \ell(Q)^{(1-n/p)p'} s(Q)^{(n-1)p'} \Bigg)^{1/p'},
\end{align*}
where in the last step we applied H\"older's inequality for $\frac{1}{p}+\frac{1}{p'}=1$. 

In order to let $\varepsilon\to 0$ and obtain the conclusion, it suffices to have
\begin{align*}
\sum_{Q\in W(D)}  \left(\frac{s(Q)}{\ell(Q)} \right)^{(n-1)p'}\ell(Q)^n <\infty,
\end{align*}
or, equivalently, $g \coloneqq \sum_{Q\in W(D)} \frac{s(Q)}{\ell(Q)}\x_Q \in L^{p'(n-1)}(D)$. Note that $g\in L^n(D)$ by Lemma \ref{Shadow-QH distance}. Furthermore, $p'(n-1)\leq n$ for $p\geq n$. Hence, the boundedness of $D$ from Corollary \ref{Holder is Quasiball} and H\"older's inequality yield the conclusion.
\end{proof}

\section{Main Theorem}\label{Section Main Theorem}

In this section we give the proof of the Main Theorem.

\begin{definition}\label{Main Holder Detour}
Let $K\subset \R^n$ be a compact set. We say that $K$ is a \textit{H\"older detour set} if $K$ is a detour set and there exist $\alpha,c>0$ such that its bounded complementary components $\{ \Omega_k\}_{k\geq 1}$ are $(\alpha,c)$-H\"older domains, and for the unbounded component $\Omega_0$ there exists a large ball $B(0,R)\supset \partial \Omega_0$ such that $\Omega_0\cap B(0,R)$ is an $(\alpha,c)$-H\"older domain. 
\end{definition}

\begin{lemma}\label{Main Detour Measure}
Let $K\subset \R^n$ be a H\"older detour set. Then $m_n(K)=0$, and $\diam(\Omega_k)\to 0$ as $k\to\infty$.
\end{lemma}
\begin{proof}
By Proposition \ref{Detour Zero measure}, it suffices to have that $m_n(\partial\Omega_k)=0$ for all $k\in \N\cup \{0\}$, and that $\sum_{k=1}^\infty \diam(\Omega_k)^n <\infty$. The first follows from Lemma \ref{Holder measure zero}, and the second from Corollary \ref{Holder is Quasiball}, which implies that 
\begin{align*}
\sum_{k=1}^\infty \diam(\Omega_k)^n\simeq \sum_{k=1}^\infty m_n(\Omega_k) \leq m_n(\R^n \setminus \Omega_0)<\infty,
\end{align*}
as $\R^n\setminus \Omega_0$ is bounded.
\end{proof}

We now restate the Main Theorem.

\begin{theorem}\label{Main Theorem}
Let $K\subset \R^n$ be a H\"older detour set, and let $p>n$. Then any continuous function $f\colon \R^n\to \R$ that lies in $ W^{1,p}(\R^n\setminus K)$, also lies in $W^{1,p}(\R^n)$.
\end{theorem}

By the remarks prior to Definition \ref{Intro Def Removability}, it suffices to prove that $f\in W^{1,p}(U)$, where $U\coloneqq B(0,R)$ is some large ball containing $K$.

Before starting the proof, we note that $f\in W^{1,p}(U)$ if and only if it has a representative that we still denote by $f$ such that:
\begin{enumerate}[\upshape(i)]
\item $f\in L^p(U)$,
\item there exist directions $v_1,\dots,v_n$ that span $\R^n$ such that for each direction $v_i$ the function $f$ is absolutely continuous in almost every line $L$ parallel to $v_i$, and
\item the partial derivative $\partial_{v_i}f$ of $f$ in the direction of $v_i$ lies in $L^p(U)$, for $i=1,\dots,n$.
\end{enumerate}
This characterization is known as Nikodym's theorem; see \cite[Section 1.1.3]{Maz}. The absolute continuity in (ii) is interpreted in the local sense, namely every point $x\in L\cap U$ has an open neighborhood in $L\cap U$ on which $f$ is absolutely continuous. If $f$ is continuous, then the above hold without the need to pass to another representative of $f$.

Assume that $K$ is as in the theorem. Note that $m_n(K)=0$ by Lemma \ref{Main Detour Measure}. Thus, in order to prove the theorem, it suffices to ``upgrade" the absolute continuity on lines $L\cap (U\setminus K)$ to absolute continuity on lines $L\cap U$. For this we will use the following basic lemma:

\begin{lemma}\label{Main Theorem Basic Lemma}
Let $f\colon (0,1)\to \R$ be a continuous function, and let $A\subset (0,1)$ be a compact set. Assume that $f$ is absolutely continuous in every interval $I\subset (0,1)\setminus A$, and that $\int_{(0,1)\setminus A} |f'| <\infty$. If $m_1(f(A))=0$, then $f$ is absolutely continuous on $(0,1)$. If $(0,1)$ is replaced by an open set $U\subset \R$ but the other assumptions are unchanged, then the conclusion is that $f$ is absolutely continuous on every component of $U$.
\end{lemma}
The lemma can be proved using the definition of absolute continuity, or it can be derived from the Banach-Zaretsky theorem \cite[Theorem 4.6.2, p.~196]{BC}.

If $L_{v_0}$ is a line parallel to a direction $v_0\in \R^n$, then $L_v\coloneqq L_{v_0}+v$, $v\in \{v_0\}^\perp $, gives the family of lines parallel to $v_0$. If $\partial_{v_0} f$ is the partial derivative of $f$ in the direction $v_0$, then 
\begin{align*}
\int_{L_v\cap (U\setminus K)} |\partial_{v_0} f|<\infty
\end{align*}
for a.e.\ $v\in \{v_0\}^\perp$, as one sees from Fubini's theorem and H\"older's inequality. The strategy of the proof of the Main Theorem is to estimate $m_1(f(L_v\cap K))$, and show that this is equal to $0$ for a.e.\ $v$ by integrating over all lines $L_v$ parallel to the direction $v_0$.  Then by Lemma  \ref{Main Theorem Basic Lemma} we will conclude that $f$ is absolutely continuous on a.e.\ line parallel to $v_0$. We also need this to be true for a set of directions $\{v_1,\dots,v_n\}$ that span $\R^n$.

We are now ready to prove the Main Theorem. We invite the reader to compare the proof with the proof of Proposition \ref{Detour Zero measure}.

\begin{proof}[Proof of Theorem \ref{Main Theorem}]
We show that $m_1(f(L\cap K))=0$ for a.e.\ line $L$ that is parallel to one of the directions $v_0$, along which the lines with the detour property exist in Definition \ref{Definition Detour set}. There are $n$ such directions that span $\R^n$, so this suffices by the preceding remarks. 

Recall here that $U$ is a large ball that contains $K$, such that $U\cap \Omega_0$ is a H\"older domain. Proposition \ref{Sobolev Holder boundary} implies that $m_1( f(L\cap \partial \Omega_k))=0$ for all $k\in \N\cup \{0\}$ and for a.e.\ line $L$, parallel to $v_0$. Thus, it suffices to prove that
\begin{align*}
m_1( f( L\cap K \setminus \bigcup_{k=0}^\infty \partial \Omega_k ))=m_1( f(L\cap U\setminus \bigcup_{k=0}^\infty\br \Omega_{k}))=0
\end{align*}
for a.e.\ line $L$ parallel to $v_0$. 

Let $Z \subset \N\cup \{0\}$ be a finite set that contains $0$. We fix a line $L$, parallel to $v_0$. The set $L\cap U \setminus \bigcup_{k\in Z}\br \Omega_k$ is an open subset of $L$, so it is the union of at most countably many open intervals. Let $\{I_j\}_{j\in \N}$ be the family of these intervals that intersect $K$, where we set $I_j =\emptyset$ for large indices $j$ if the intervals are finitely many. Note for each $j\in \N$ the interval $\br I_j$ does not intersect $\partial U\subset \Omega_0$, and has its endpoints on sets $\partial \Omega_k$, $k\in Z$. This is because the index $k=0$ is already contained in $Z$.

\begin{figure}
\begin{tikzpicture}[line cap=round,line join=round,>=triangle 45,x=1.0cm,y=1.0cm]

	\clip(7.45,-6.26) rectangle (19,0.79);
	\fill[fill=black!5] (7.45,-6.26) rectangle (21,0.79);
	
	\draw[fill=white](2.14,0.3) circle (8.226809831277249cm);
	\draw[fill=white](23.32,-2.16) circle (6.384982380555172cm);
	\draw[fill=white](11.,-2.) circle (0.9280684951062352cm);
	\draw[fill=white](12.66,-1.16) circle (0.9441912862916136cm);
	\draw[fill=white](13.82,-3.44) circle (1.619459061491593cm);
	\draw[fill=white](11.62,-3.4) circle (0.5942429088529437cm);
	\draw[fill=white](12.143423138930226,-2.4853261934543767) circle (0.31150552182162605cm);
	\draw[fill=white](15.558220820272451,-1.6844141573186742) circle (0.8511781025342522cm);
	\draw[fill=white](16.32272776385653,-2.448921100902754) circle (0.23028817764831078cm);
	\draw[fill=white](16.664935633841786,-2.208647490062043) circle (0.19618527246605788cm);
	\draw[fill=white](16.664935633841786,-2.514450267495675) circle (0.11291534817782811cm);
	\draw[fill=white](16.879248577270303,-2.3443489415526373) circle (0.05753439006522982cm);

	\draw [line width=0.8pt,color=green!50!black] (9.95,-2.32)-- (16.95, -2.32);
	\draw [line width=0.8pt,color=green!50!black,dashed] (0,-2.32)--(9.95,-2.32);
	\draw [line width=0.8pt,color=green!50!black,dashed] (16.95,-2.32)--(20,-2.32);
	
	\node[draw, shape=circle, fill=black, scale=.15,label=below:$x$] (x) at (10.05,-2.32) {};
	\node[draw, shape=circle, fill=black, scale=.15,label=above right:$y$] (y) at (16.95,-2.32) {};
	\node[draw, shape=circle, fill=black, scale=.15,label=above:$a$] (a) at (11.1,-2.32) {};
	\node[draw, shape=circle, fill=black, scale=.15,label=above:$b$] (b) at (15.8,-2.32) {};
	\node[draw, shape=circle, fill=black, scale=.15,label=above:$c$] (c) at (11,-1.7) {};
	\node[draw, shape=circle, fill=black, scale=.15,label=above:$d$] (d) at (15.7,-1.6) {};
	
	\node[draw, shape=circle, fill=black, scale=.15] (n1) at (10.12,-1.7) {};
	\node[draw, shape=circle, fill=black, scale=.15] (n2) at (11.84,-2.38) {};
	\node[draw, shape=circle, fill=black, scale=.15] (n3) at (12.4,-2.63) {};
	\node[draw, shape=circle, fill=black, scale=.15] (n4) at (14.95,-2.3) {};
	\node[draw, shape=circle, fill=black, scale=.15] (n5) at (16.18,-2.28) {};

	\draw[thick,color=purple,dashed] (7,-2).. controls (8,-1.2).. (n1);
	\draw[thick,color=purple,dashed] (n1).. controls (10.6,-1.8).. (c);
	\draw[thick,color=purple] (c).. controls (11.4,-1.6) and (11.6,-2.4).. (n2);
	\draw[thick,color=purple] (n2).. controls (12.2,-2.3) and (12.3,-2.6).. (n3);
	\draw[thick,color=purple] (n3).. controls (13,-3) and (14.4,-2.6).. (n4);
	\draw[thick,color=purple] (n4).. controls (15.2,-2.2)and (15.4,-1.4).. (d);
	\draw[thick,color=purple,dashed] (d).. controls (16,-2).. (n5);
	\draw[thick,color=purple,dashed] (n5).. controls (16.4,-2.6)and (16.7,-2.1).. (y);
	\draw[thick,color=purple,dashed] (y).. controls (17.4,-3)and (18.1,-2).. (19,-2.7);

	\node[anchor=north] (gammaI) at (13.6,-2.8) {$\gamma_I$};	
	\node[anchor=south] (gamma) at (9,-1.4) {$\gamma$};
	\node[anchor=north] (L) at (9,-2.3) {$L$};
	\node[anchor=north west] (Omegak) at (8,0) {$\Omega_k$};
	\node[anchor=north west] (Omegal) at (18,0) {$\Omega_l$};
	

	\end{tikzpicture}
\caption{The detour path $\gamma$ and its subpath $\gamma_I$.}\label{fig:detour}
\end{figure}
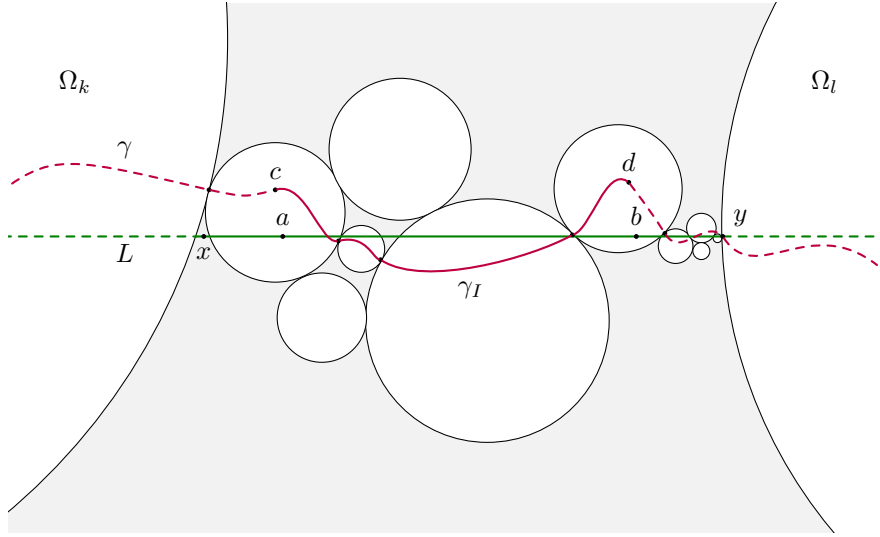

We first fix one interval $I=I_j$ and we estimate $m_1(f(I\cap K))$. Since $f$ is uniformly continuous in a neighborhood of $I$, for a fixed $\varepsilon>0$ there exists $\delta>0$ such that $|f(z)-f(w)|<\varepsilon$ whenever $|z-w|<\delta$, and $z,w$ lie in the $\delta$-neighborhood of $I$. Let $x,y\in \br I\cap K$ be such that $|f(x)-f(y)|=\diam(f(I\cap K))$. Consider a compact subinterval $[a,b] \subset (x,y)$ such that $|x-a|<\delta/4$ and $|y-b|<\delta/4$. Let $\gamma$ be a detour path as in Definition \ref{Definition Detour set} that lies $\eta$-close to the line $L$, where $0<\eta<\delta/4$,  and $\eta$ is so small that the $\eta$-neighborhood of $[a,b]$ does not intersect the closed set $\bigcup_{k\in Z} \br\Omega_k$. Then we can choose a subpath $\gamma_I$ of $\gamma$ that is entirely contained in the $\eta$-neighborhood of $[a,b]$, and whose endpoints $c,d$ satisfy $|a-c|<\eta<\delta/4$ and $|b-d|<\eta<\delta/4$; see Figure \ref{fig:detour}. In particular, $\gamma_I $ does not intersect $\bigcup_{k\in Z} \br\Omega_k$. Note that $|c-x|<\delta/2$ and $|d-y|<\delta/2$. Now, using the uniform continuity we have
\begin{align*}
|f(x)-f(y)| \leq 2\varepsilon  + |f(c)-f(d)|.
\end{align*}
Recall that $\gamma_I$ is contained in the union of  finitely many sets $\br \Omega_{k}$, $k\in \Gamma(I)$, $k\notin Z$;  here $\Gamma(I)=\{k\in \Gamma : \gamma_I \cap \br\Omega_k\neq \emptyset\}$ and $\Gamma$ is as in Definition \ref{Definition Detour set}(ii), so that $\gamma\subset \bigcup_{k\in \Gamma} \br \Omega_k$. If $c,d$ lie on boundaries of sets $\Omega_k$, then using the fact that $\Gamma(I)$ is a finite set and the triangle inequality we have the estimate
\begin{align*}
|f(c)-f(d)| \leq \sum_{k\in \Gamma(I)} \diam(f(\partial\Omega_k)). 
\end{align*} 
If $c\in \Omega_k$, then $\dist(c,\partial \Omega_k)<|c-x| <\delta/2$ (since $x\in K$), and a similar inequality is true for $d$ in case it lies in some set $\Omega_k$. By the uniform continuity we thus have
\begin{align*}
|f(c)-f(d)|\leq  2\varepsilon +\sum_{k\in \Gamma(I)} \diam(f(\partial\Omega_k)).
\end{align*} 
Summarizing, we have the estimate
\begin{align}\label{Main Theorem One interval bound}
\diam(f(I\cap K)) \leq 4\varepsilon +\sum_{k\in \Gamma(I)} \diam(f(\partial\Omega_k)).
\end{align}

We now fix finitely many intervals $I_1,\dots,I_N$ and run the same procedure for each individual interval to obtain points $x_j,y_j\in \br I_j$ such that $\diam(f(I_j\cap K))=|f(x_j)-f(y_j)|$, and  paths $\gamma_j= \gamma_{I_j}$ as above, that do not intersect $\bigcup_{k\in Z}\br\Omega_k$, and are contained in the union of finitely many regions $\br\Omega_k$, $k\in \Gamma(I_j)$. We note that $\varepsilon$ is still fixed. By running a greedy algorithm we group the indices $j$ into finite sets $A_i$, $i=1,\dots,M$, such that:
\begin{enumerate}[(i)]
\item intervals $I_j$ with indices $j$ lying in distinct sets $A_i$ intersect disjoint sets of regions $\br\Omega_k$, $k\in \bigcup_{j=1}^N \Gamma(I_{j})$, and
\item $\bigcup_{j\in A_i}\bigcup_{k\in \Gamma(I_j)} \br \Omega_k$ is a connected set for all $i=1,\dots,M$.
\end{enumerate}

We fix $i$ and points $z,w\in \bigcup_{j\in A_i}I_j\cap K$. Assume that $z\in I_{j_1}$ and $w\in I_{j_2}$. We have
\begin{align*}
|f(z)-f(w)|&\leq  |f(z)-f(x_{j_1})| + |f(x_{j_1}) -f(x_{j_2})|+ |f(x_{j_2})-f(w)|\\
&\leq  \diam( f(I_{j_1}\cap K)) + |f(x_{j_1}) -f(x_{j_2})| + \diam(f (I_{j_2}\cap K)).
\end{align*}	
We bound the first and last term using \eqref{Main Theorem One interval bound}. For the middle term we employ the same procedure that we used to derive \eqref{Main Theorem One interval bound}; here we need property (ii) of the set $A_i$ to ``travel" from $x_{j_1}$ to $x_{j_2}$ using sets $\br \Omega_{k}$ that are intersected by $\gamma_j$ for $j\in A_i$. Thus, we get the estimate
\begin{align*}
|f(x_{j_1}) -f(x_{j_2})| \leq 4\varepsilon + \sum_{{k\in \bigcup_{j\in A_i}\Gamma(I_j)}} \diam(f(\partial \Omega_k)).
\end{align*}
Combining the above we obtain 
\begin{align*}
|f(z)-f(w)|&\leq 12\varepsilon + \sum_{k\in \Gamma(I_{j_1})} \diam(f(\partial\Omega_k))+\sum_{k\in \Gamma(I_{j_2})} \diam(f(\partial\Omega_k))\\
&\qquad \qquad+\sum_{{k\in \bigcup_{j\in A_i}\Gamma(I_j)}} \diam(f(\partial \Omega_k))\\
&\leq 12\varepsilon + 3\cdot \sum_{{k\in \bigcup_{j\in A_i}\Gamma(I_j)}} \diam(f(\partial \Omega_k)).
\end{align*}
This yields the bound
\begin{align*}
\diam(f(\bigcup_{j\in A_i}I_j\cap K)) \leq 12\varepsilon + 3\cdot \sum_{{k\in \bigcup_{j\in A_i}\Gamma(I_j)}} \diam(f(\partial \Omega_k)).
\end{align*}

Hence, we have 
\begin{align*}
m_1( f(\bigcup_{j=1}^N I_j\cap K))&=m_1( \bigcup_{i=1}^M f(\bigcup_{j\in A_i}I_j\cap K))\\
&\leq \sum_{i=1}^M \diam(f(\bigcup_{j\in A_i}I_j\cap K))\\
&\leq 12M\varepsilon +3 \sum_{i=1}^M \sum_{{k\in \bigcup_{j\in A_i}\Gamma(I_j)}} \diam(f(\partial \Omega_k)).
\end{align*}
Recall here that the detour paths $\gamma_j$ do not intersect $\br\Omega_k$, $k\in Z$, and 
\begin{align*}
\bigcup_{j=1}^N \Gamma(I_j) \subset \Gamma\subset  \{k: \br \Omega_k\cap L \neq \emptyset\}\setminus Z,
\end{align*}
by the definition of the paths $\gamma_j$. Using property (i) of the sets $A_i$ we can write the above estimate as a single sum:
\begin{align*}
m_1( f(\bigcup_{j=1}^N I_j\cap K)) \leq  12M\varepsilon +3 \sum_{\substack{k:\br \Omega_k\cap L\neq \emptyset\\ k\notin Z}} \diam(f(\partial \Omega_k)).
\end{align*}
Note that $\varepsilon$ was arbitrary, so letting $\varepsilon \to 0$ we have
\begin{align*}
m_1( f(\bigcup_{j=1}^N I_j\cap K)) \leq 3 \sum_{\substack{k:\br \Omega_k\cap L\neq \emptyset\\ k\notin Z}} \diam(f(\partial \Omega_k)).
\end{align*}
Finally, letting $N\to\infty$ we obtain
\begin{align*}
m_1( f(L\cap K\setminus \bigcup_{k=0}^\infty \br\Omega_k)) &\leq m_1( f(L\cap K \setminus \bigcup_{k\in Z} \br \Omega_k))\\
&= m_1( f(\bigcup_{j=1}^\infty I_j\cap K))\\
&\leq 3 \sum_{\substack{k:\br \Omega_k\cap L\neq \emptyset\\ k\notin Z}} \diam(f(\partial \Omega_k)).
\end{align*}

To complete the proof we integrate over all lines $L=L_v=L_{v_0}+v$, $v\in \{v_0\}^\perp \simeq \R^{n-1}$,  and using Fubini's theorem we have
\begin{align*}
\int m_1( f(L_v\cap K \setminus \bigcup_{k=0}^\infty \br \Omega_k)) \, dm_{n-1}(v)&\leq 3 \sum_{k\notin Z} \diam(f(\partial \Omega_k)) \int_{\br \Omega_k\cap L_v\neq \emptyset} dm_{m-1}(v)\\
&\leq 3\sum_{k\notin Z} \diam(f(\partial \Omega_k))  \diam (\Omega_k)^{n-1}.
\end{align*}
If the last sum is convergent, by letting $Z\to \N\cup \{0\}$ we obtain that 
\begin{align*}
m_1( f(L_v\cap K \setminus \bigcup_{k=0}^\infty \br \Omega_k))=0
\end{align*}
for a.e.\ $v\in \{v_0\}^\perp$, as desired.

To prove our claim we use Proposition \ref{Sobolev Holder estimate}, applied to $f|_{\br\Omega_k}$, $k\geq 1$. We have
\begin{align*}
\sum_{k\notin Z} \diam(f(\partial \Omega_k))  \diam (\Omega_k)^{n-1} &\lesssim \sum_{k=1}^\infty \diam(\Omega_k)^n \left(\mint{-}_{\Omega_k} |\nabla f|^p \right)^{1/p} \\
&\simeq \sum_{k=1}^\infty m_n(\Omega_k)^{1-1/p} \left(\int_{\Omega_k} |\nabla f|^p \right)^{1/p},
\end{align*}
since $m_n(\Omega_k)\simeq \diam(\Omega_k)^n$ by Corollary \ref{Holder is Quasiball}, with uniform constants. Applying H\"older's inequality for the exponents $\frac{1}{p}+\frac{1}{p'}=1$, we see that the latter expression is bounded by 
\begin{align*}
\left(\sum_{k=1}^\infty  m_n(\Omega_k) \right)^{1/p'} \left( \int_U |\nabla f|^p\right)^{1/p}.
\end{align*}
Noting that  $\sum_{k=1}^\infty m_n(\Omega_k)\leq  m_n(\R^n\setminus \Omega_0) <\infty $, yields the conclusion.
\end{proof}

\section{Constructions and Examples}\label{Section Examples}

\subsection{A construction}
We give a general construction of detour sets in the plane.

\begin{prop}\label{Example Construction}
Let $K_m\subset K_{m-1}\subset \R^2$, $m\in \N$, be a sequence of compact sets with the property that for each $m\in \N\cup \{0\}$ the set $\inter(K_m)$ is the union of finitely many Jordan regions $\{V_{i,m}\}_{i\in I_m}$, and also $\R^2 \setminus K_m$ has finitely many connected components $\{\Omega_{j,m}\}_{j\in J_m}$, which are Jordan regions. Moreover, suppose that for each $m\in \N$ and $i\in I_m$ the boundary $\partial V_{i,m}$ is the union of at most three arcs, each contained in a boundary $\partial \Omega_{j,m'}$, where $j\in J_{m'}$ and $m'\leq m$. Finally, assume that the diameters of $V_{i,m}$ shrink to $0$ as $m\to\infty$, i.e.,
\begin{align*}
\lim_{m\to\infty }\sup_{i\in I_m} \diam(V_{i,m})=0.
\end{align*} 
Then $K\coloneqq \bigcap_{m=1}^\infty K_m$ is a detour set, if it has infinitely many complementary components.
\end{prop}

The Sierpi\'nski gasket is a prototype for this construction.

\begin{proof}
Let $L \subset \R^2$ be a line intersecting $K$ and fix $\varepsilon>0$. We shall show that there exists a detour path $\gamma$ lying in the $\varepsilon$-neighborhood of $L$. We fix a large $m\in \N$ such that $\diam( V_{i,m})<\varepsilon$ for all $i\in I_m$.  The set $L\setminus  \inter(K_m)$ is contained in $\bigcup_{j\in J_m} \br \Omega_{j,m}$, so it also is contained in the union of the closures of finitely many complementary components of $K$,  since $K\subset K_m$. In order to construct a detour path $\gamma$, we need to create small detours for the parts of $L$ intersecting $\inter(K_m)$. 

More specifically, we claim that we can cover $L\cap \inter(K_m)$ by finitely many closed intervals $[x_l,y_l]$, such that each interval has both of its endpoints on some set $\partial V_{i,m}$. Assume the claim for the moment. By assumption, $\partial V_{i,m}$ is the union of at most three arcs (suppose exactly three for the sake of the argument), each contained in a boundary $\partial \Omega_{j_1,m_1}, \partial \Omega_{j_2,m_2}$, $\partial \Omega_{j_3,m_3}$, respectively.  Since $\partial V_{i,m}$ is a Jordan curve containing $x_l,y_l$, there exists one arc $A_l$ of $\partial V_{i,m}$ connecting $x_l$ and $y_l$ that is contained in one or in the union of two of the arcs  $\partial \Omega_{j_1,m_1}, \partial \Omega_{j_2,m_2}$, $\partial \Omega_{j_3,m_3}$; namely $A_l$ is contained in the union of an arc that contains $x_l$ and an arc that contains $y_l$. Hence, the arc $A_l$ is contained in one or two components $\br \Omega_{j,m'}$ that intersect the line $L$. 

Replacing each of the intervals $[x_l,y_l]$ with an arc $A_l$ of a corresponding set $\partial V_{i,m}$, which has diameter less than $\varepsilon$, we obtain the desired detour path that lies $\varepsilon$-close to $L$. Finally we note that by construction the detour path is contained in the closures of finitely many complementary components of $K$, all of which are intersected by $L$, so the line $L$ has the detour property as required in Definition \ref{Definition Detour set}(ii) and (iii).

The construction of the intervals $[x_l,y_l]$ is done as follows. Assume that $L=\R$ and in particular it has the standard ordering. Then it meets the components $V_{i,m}$, $i\in I_m$, in some order, with respect to the point of first entry. We let $x_1\in \partial V_{i,m}$ be the first entry point of $L$ into some component $V_{i,m}$ of $\inter(K_m)$, and $y_1\in \partial V_{i,m}$, $y_1\geq x_1$, be the last exit point of $L$ from $V_{i,m}$. Then we let $x_2\geq y_1$ be the first entry point of $L$ into the ``next" component $V_{i',m}$, and $y_2\geq x_2$ be the last exit point. We proceed inductively.
\end{proof}

A construction ``dual" to this one is the following. Again, a good example of this construction to have in mind is the Sierpi\'nski gasket. A non-empty set $C\subset \R^2$ is called \textit{non-trivial} if it is not a point.

\begin{prop}\label{Example Construction2}
Let $K_0\subset \R^2$ be a compact set such that $\inter(K_0)$ consists of finitely many disjoint Jordan regions $V_{i,0}$, $i\in I_0$, and $\R^2\setminus K_0$ is the union of at least two, but finitely many disjoint Jordan regions $\Omega_{j,0}$, $j\in J_0$. Once $K_{m-1}$ is defined, we define $K_{m}$ by removing from each component $V_{i,m-1}$, $i\in I_{m-1}$, of $\inter(K_{m-1})$ a Jordan region $\Omega_{j,m}$, $j\in J_m$, in such a way that
\begin{enumerate}[\upshape(i)]
\item $\partial \Omega_{j,m}$ intersects an open arc contained in each non-trivial component of $\partial V_{i,m-1}\cap \partial \Omega_{j',k}$, $k\in \{0,\dots,m-1\}$, $j'\in J_k$, whenever the latter intersection is non-empty, and
\item $\inter(K_m)$ has finitely many components.
\end{enumerate}
Furthermore, assume that 
\begin{align*}
\lim_{m\to\infty} \sup_{j\in J_m} \diam(\Omega_{j,m}) =0.
\end{align*} 
Then the set $K\coloneqq \bigcap_{m=0}^\infty K_m$ is a detour set.
\end{prop}
\begin{proof}
First we prove that the inductive construction is valid. For this, we need to check that each component $V_{i,m}$ of $\inter(K_m)$ is a Jordan region for all $m\geq 0$. We do this for $K_1$.

We claim that the boundary of each component $ V_{i,0}$ of $\inter (K_{0})$ has to contain  non-trivial arcs of at least two distinct boundaries $\partial \Omega_{j,0},\partial \Omega_{j',0}$. To see this, note first that  $\partial V_{i,0}$ is contained in the union $\bigcup_{j\in J_0} \partial \Omega_{j,0}$. If the claim were false, then there would exist $j\in J_0$ such that for all $j'\neq j$, $j'\in J_0$, the set $\partial V_{i,0}\cap \partial \Omega_{j',0}$ is either empty or a totally disconnected compact set, and thus $\partial V_{i,0}\subset \partial \Omega_{j,0}$. In fact, $\partial V_{i,0}=\partial \Omega_{j,0}$ because they are both Jordan curves. This implies that $\Omega_{j,0}$ is an unbounded Jordan region, and that $K_0= \br V_{i,0}$, which contradicts the assumption that $\R^2 \setminus K_0$ consists of at least two Jordan regions.

Assume now that $\Omega_{j,1}$ is removed from $V_{i,0}$ in the construction of $K_1$. Then, by construction,	 $\partial \Omega_{j,1}$ has to intersect an open arc  from each non-trivial component of $\partial \Omega_{j',0}\cap \partial V_{i,0}$,  $j'\in J_0$. The preceding paragraph implies that at least two such non-trivial components exist, and therefore $\partial \Omega_{j,1}$  intersects $\partial V_{i,0}$ in at least two points. We now use the following \textit{general fact}, which follows from Ker\'ekj\'art\'o's theorem \cite[Chapter VI.16, p.~168]{New}:

 \begin{flushleft}
\textit{If $\Omega, V$ are bounded Jordan regions with $\Omega\subset V$, and $\partial \Omega\cap \partial V$ contains at least two points, then each component of $V\setminus \br \Omega$ is a Jordan region, whose boundary is the union of a non-trivial arc of $\partial\Omega$ and a non-trivial arc of $\partial V$.}
 \end{flushleft}
 
In our case, this implies that each  component $V_{i',1}$ of $V_{i,0}\setminus \br \Omega_{j,1}$ is a Jordan region, whose boundary $\partial V_{i',1}$ is the union of a non-trivial arc of $\partial\Omega_{j,1}$ and a non-trivial arc of $\partial V_{i,0} \subset \bigcup_{j'\in J_0}\partial \Omega_{j',0}$. In particular, $\partial V_{i',1}$ also contains a non-trivial arc of some $\partial \Omega_{j',0}$. Summarizing, $\partial V_{i',1}$ contains a non-trivial arc of $\partial \Omega_{j,1}$ and a non-trivial arc of $\partial \Omega_{j',0}$. Now, the same argument can be run with $V_{i,0}$ replaced by $V_{i',1}$, and one proceeds inductively to prove that $V_{i,m}$ is a Jordan region for all $i,m$.

Note at this point that $\partial V_{i',1}$ cannot be entirely contained in the boundary of one set $\partial \Omega_{l,m}$, but it is the union of two or three non-trivial arcs, each contained in a distinct set $\partial \Omega_{l,m}$, $m\in \{0,1\}$, $l\in J_m$. Indeed, one of them is a subarc of $\partial \Omega_{j,1}$. The other ones are precisely arcs of sets $\partial \Omega_{l,0}$ whose intersection with $\partial V_{i',1}\cap \partial V_{i,0}$  is a non-trivial arc. By preceding arguments, there exists at least one such set. However, there can be at most two. That is because if there were a third one, there would exist $l\in J_0$ and a non-trivial component of $\partial V_{i,0}\cap \partial \Omega_{l,0}$ whose interior arc is not intersected by $\partial \Omega_{j,1}$, contradicting (i); see Figure \ref{fig-construction}. By induction, for all $m\geq 1$, $i\in I_m$, the set $\partial V_{i,m}$ cannot be entirely contained in the boundary of one set $\partial \Omega_{j,k}$, and it is the union of two or three non-trivial arcs, each contained in a distinct set $\partial \Omega_{j,k}$, $k\in \{0,\dots,m\}$, $j\in J_k$.

\begin{figure}
\centering
\begin{tikzpicture}
	
	//left picture (possible)		
	\begin{scope}[shift={(-3,0)},scale=.8]			
			//grid  
			//grid	
 					
 			//Omega, V at level 0
 			\draw[thick, fill=red!10] (-3,-5)--(1,-5)--(-3,-4)--(-3,-5);
 			\draw[thick, fill=red!10] (-3,-4)--(1,-4)--(-3,-3)--(-3,-4);
			\draw[thick, fill=red!10] (-3,-3)--(-2.5,-2.5)--(-3,-2)--(-3,-3); 			
 			\draw[thick,fill=red!10] (-3,-2)--(1,-4)--(-3,-1)--(-3,-2);
 			\draw[thick, fill=red!10] (-3,-1)--(1,0)--(-3,0)--(-3,-1);

 			//big square
 			\draw[thick, fill=red!10] (1,-5)--(5,-5)--(5,0)--(1,0)--(1,-5);
 			//V 1
 			\filldraw[pattern=north west lines, pattern color=red!40] (1,-3)--(3,-4)--(3,-4.5)--(1,-4.7)--(1,-3);
 			//Omega 1
 			\draw[very thick, color= blue, fill=green!10] (1,-1)--(1,-1.5)--(2,-2)--(1,-3)--(3,-4)--(3,-4.5)--(1,-4.7)--(4,-5)-- (5,-3)--(5,-2)--(4,0)--(2,0)--(1,-1);

 			//nodes
 			\node[label=$\Omega_{j,1}$] at (3.5,-3){};
 			\node[label=$V_{i',1}$] at (2,-4.5){};
 			\node[label=$\Omega_{l,0}$] at (0,-2.5){};
 			\node[label=$\partial V_{i,0}$] at (5.6,-5){};
 	\end{scope}
 	
 	//right picture (impossible)
 	\begin{scope}[shift={(3.5,0)},scale=.8]
 	//grid  
			//grid	
 					
 			\clip(-1,-3) rectangle (4,-5);
 			//Omega, V at level 0
 			\draw[thick, fill=red!10] (-3,-5)--(1,-5)--(-3,-4)--(-3,-5);
 			\draw[thick, fill=red!10] (-3,-4)--(1,-4)--(-3,-3)--(-3,-4);
			\draw[thick, fill=red!10] (-3,-3)--(-2.5,-2.5)--(-3,-2)--(-3,-3); 			
 			\draw[thick,fill=red!10] (-3,-2)--(1,-3.7)--(-3,-1)--(-3,-2);
 			\draw[thick, fill=red!10] (-3,-1)--(1,0)--(-3,0)--(-3,-1);
 			
 			//big square
 			\draw[thick, fill=red!10] (1,-5)--(5,-5)--(5,0)--(1,0)--(1,-5);
 			//V 1
 			\filldraw[pattern=north west lines, pattern color=red!40] (1,-3)--(3,-4)--(3,-4.5)--(1,-4.7)--(1,-3);
 			//Omega 1
 			\draw[very thick, color= blue, fill=green!10] (1,-1)--(1,-1.5)--(2,-2)--(1,-3)--(3,-4)--(3,-4.5)--(1,-4.7)--(4,-5)-- (5,-3)--(5,-2)--(4,0)--(2,0)--(1,-1);

 			//nodes
 			\node[label=$\Omega_{j,1}$] at (3.5,-3){};
 			\node[label=$V_{i',1}$] at (2,-4.5){};
 			\node[scale=.5, label=$\Omega_{l,0}$] at (0,-3.95){};
 			\node[label=$\partial V_{i,0}$] at (5.5,-5){};
 	\end{scope}	
	\end{tikzpicture}

\caption{In the left figure, pink color denotes the set $K_0$, after the removal of $\Omega_{j,1}$ from $V_{i,0}$. Here, $\partial V_{i',1}$ is the union of three arcs of sets $\partial \Omega_{k,m}$. On the right, we see an impossible situation.}\label{fig-construction}
\end{figure}

In order to prove that $K$ is a detour set, by Proposition \ref{Example Construction}, it suffices to prove that the diameters of $V_{i,m}$ converge to $0$, under our assumptions.

We argue by contradiction, assuming that for some $\varepsilon>0$ there exists a sequence $V_k$, $k\in \N$, of components of $\inter(K_{m_k})$, $m_k\nearrow \infty $, such that $\diam(V_k)\geq \varepsilon$ for all $k$. Since the diameters of $\Omega_{j,m}$ converge to $0$, and $\partial V_{k}$ is the union of two or three arcs of sets $\partial \Omega_{j,m}$, we conclude that there exists some set $\Omega_{j_1,m_1}\eqqcolon \Omega_1$ such that $A_k\coloneqq \partial \Omega_1 \cap \partial V_k \neq \emptyset$, and $\diam(A_k)\geq \varepsilon/3 $ for infinitely many $k\in \N$. By passing to a subsequence, we assume that this holds for all $k$, and that $A_k$ intersects the same non-trivial segment of $\partial \Omega_1$ for all $k$.

From the construction we see that this is only possible if $A_{k}\subset A_{k-1}$, and $V_k\subset V_{k-1}$. Thus, the endpoints $x_k,y_k$ of $A_k$ converge to distinct points $x,y\in \partial \Omega_1$. This shows that there exists $\delta>0$ such that for all large enough $k$ the points $x_k,y_k$ are at least $\delta$ apart. Hence, the complementary arc of $A_k$ in $\partial V_k$ has diameter at least $\delta$. Using again the fact that there are finitely many ``large" sets $\Omega_{j,m}$ and that $\partial V_k$ is the  union of two or three non-trivial arcs of distinct sets $\partial \Omega_{j,m}$, we see that there exists $\Omega_{j_2,m_2}\eqqcolon \Omega_2$, distinct from $\Omega_1$, such that $B_k\coloneqq \partial \Omega_2 \cap \partial V_k \neq \emptyset$, and $\diam(B_k)\geq \delta/3$ for infinitely many $k\in \N$. Using a subsequence, we may assume that this holds for all $k$. Note that $B_{k}\subset B_{k-1}$, and that $B_k$ is adjacent to $A_k$, i.e., $A_k\cap B_k$ contains at least one point (and at most two).

Summarizing, we have that the Jordan curve $\partial V_k$ contains two adjacent arcs $A_k \subset \partial \Omega_1$ and $B_k\subset \partial \Omega_2$ with diameters bounded below by some $\varepsilon_0>0$ for all $k\in \N$, and such that $A_k$ and $B_k$ are  decreasing in $k$.

Assume that $A_k$ and $B_k$ share only one endpoint for infinitely many $k$, and thus for all $k$ after passing to a subsequence. In this case, the set $\partial V_k$ has to be the union of three non-trivial arcs of  distinct sets $\partial \Omega_{j,m}$, two of which are already $A_k$ and $B_k$. We denote by  $C_k$ the third arc connecting the other endpoints of $A_k$ and $B_k$, so that $\partial V_k =A_k\cup B_k\cup C_k$. For each $k$ the arc $C_k$ is contained in some $\partial \Omega_k\coloneqq \partial \Omega_{j_k,m_k}$. The sets $\Omega_k$ have to be distinct, since $V_k$ is obtained when we remove from the previous level of the construction a ``new" Jordan region $\Omega_{j,m}$ (see (i)), whose boundary has to intersect $\partial V_k$ by the \textit{general fact} mentioned in the beginning of the proof. If $z_k\in A_k$, $w_k\in B_k$ are the endpoints of $C_k$, then they converge to distinct points $z,w$, because the diameters of $A_k$ and $B_k$ are bounded below. This implies that $\diam(C_k)$, and thus $\diam(\Omega_k)$, does not converge to $0$, a contradiction. 

If $A_k$ and $B_k$ share both endpoints for all but finitely many $k$, then $\partial V_k=A_k\cup B_k$, say, for $k\geq k_0$. We fix $k\geq k_0$ for the moment. In the next level of the construction, according to (i), we have to remove a Jordan region $\Omega_{j,m}$ from $V_k$, such that $\partial \Omega_{j,m}$ intersects both $A_k$ and $B_k$ (at ``interior" points). Thus, using the \textit{general fact}, we see that a ``descendant" $V_k'$ of $V_k$ contains in its boundary a non-trivial arc of a ``new" Jordan curve $\partial \Omega_{j,m}$, distinct from $\partial \Omega_1\supset A_k$ and $\partial \Omega_2\supset B_k$. The same is true for all further ``descendants" of $V_k$, hence, also for $V_{k+1} \subset V_k$. This shows that $V_{k+1}$ is the union of three non-trivial arcs:  $A_{k+1}, B_{k+1}$, and a subarc of some $\partial \Omega_{j,m}$. We have already reached a contradiction, since $V_{k+1}=A_{k+1}\cup B_{k+1}$ in this case.
\end{proof}

\begin{remark}\label{Example Homeo invariance}
If $K$ is constructed inductively as in Proposition \ref{Example Construction} or Proposition \ref{Example Construction2}, then any homeomorphism $h\colon \R^2 \to \R^2$ yields a set $K'=h(K)$ that can also be constructed inductively, by pushing forward the construction of $K$, via $h$. Furthermore, the assumptions on the shrinking diameters remain invariant under homeomorphisms, so if $K$ is a detour set then the set $K'$ is also a detour set.
\end{remark}

\subsection{Examples}

\subsubsection{The Sierpi\'nski gasket}
We recall the construction of the Sierpi\'nski gasket from the Introduction. Let $K_0$ be  an equilateral triangle of sidelength $1$, and subdivide it into four equilateral triangles of sidelength $1/2$,  with disjoint interiors. Let $K_1$ be $K_0$ with the open middle triangle removed, so $K_1$ is the union of three closed equilateral triangles of sidelength $1/2$. In the $m$-th step the set $K_m\subset K_{m-1}$ is obtained by subdividing each of the $3^{m-1}$ triangles of $K_{m-1}$ into four equilateral triangles of sidelength $1/2^m$, and removing the middle triangles.  The set $K\coloneqq \bigcap_{m=1}^\infty K_m$ is  the {Sierpi\'nski gasket}.

By Proposition \ref{Example Construction} we see that $K$ has the detour property. Furthermore, all bounded components of $\R^2 \setminus K$ are equilateral triangles, which are uniform H\"older domains. The unbounded component is also a H\"older domain (in the sense of Definition \ref{Main Holder Detour}). Thus, $K$ is a H\"older detour set.  By Theorem \ref{Intro main theorem}, $K$ is $W^{1,p}$-removable for $p>2$.

\subsubsection{Apollonian gaskets}
An Apollonian gasket is constructed inductively as follows; see Figure \ref{fig:carpet-Apollonian}. Let $C_1,C_2,C_3$ be three mutually tangent circles in the plane with disjoint interiors, and let $\Omega_1,\Omega_2,\Omega_3$ be the disks that they enclose, respectively. Then by a theorem of Apollonius there exist exactly two circles that are tangent to all three of $C_1,C_2,C_3$. We denote by $C_0$ the outer circle that is tangent to $C_1,C_2,C_3$, and by $\Omega_0$ the unbounded region in the exterior of $C_0$. For the inductive step we apply Apollonius's theorem to all triples of  mutually tangent circles of the previous step. In this way, we obtain a countable collection of circles $\{C_k\}_{k\geq 0}$, and Jordan regions $\{\Omega_k\}_{k\geq 0}$. The set  $K\coloneqq \R^2 \setminus \bigcup_{k=0}^\infty \Omega_k $ is an  \textit{Apollonian gasket}. 

We will show that $K$ is a detour set using Proposition \ref{Example Construction2}. Define $K_0= \R^2 \setminus \bigcup_{k=0}^3 \Omega_k$, and $K_m=\R^2 \setminus \bigcup_{k=0}^{m+3} \Omega_k$. Then $K=  \bigcap_{m=0}^\infty K_m$. Observe that every three mutually tangent circles split their complement in the sphere $\widehat{\C}$ in two components which are Jordan regions. The set $\inter(K_0)$ has finitely many components which are Jordan regions, and the interior of its complement consists of four disjoint Jordan regions. Inductively, $\inter(K_m)$ has finitely many components, which are all Jordan regions for $m\geq 0$.  Furthermore $K_{m+1}$ is obtained from $K_m$ exactly as in the setting of Proposition \ref{Example Construction2}. Namely, from each component $V$ of $\inter(K_{m})$, we remove a disk that is tangent to all three circles comprising the boundary of $V$. Finally, the diameters of the circles that we remove have to converge to zero. Indeed, if $r_k$ is the radius of $C_k$, then 
\begin{align*}
\pi \sum_{k=1}^\infty r_k^2= \sum_{k=1}^\infty m_2(\Omega_k) \leq m_2(\R^2\setminus \Omega_0)<\infty.
\end{align*}
Hence, $K$ is indeed a detour set. 

All bounded components of $\R^2\setminus K$ are disks, so they are uniform H\"older domains. The unbounded component is also a H\"older domain. Thus, $K$ is $W^{1,p}$-removable for $p>2$, by Theorem \ref{Intro main theorem}.

\subsubsection{Julia sets}

Detour sets also appear in the setting of Complex Dynamics as Julia sets of certain types of rational maps. Let $f\colon \widehat{\C} \to \widehat{\C}$ be a rational map of degree at least $2$. We denote by $f^n$ the $n$-fold composition $f\circ\dots\circ f$. The \textit{Julia set} $\mathbf J(f)$ of $f$ is the set of points  $z\in \widehat{\C}$ that have arbitrarily small open neighborhoods $U\subset \widehat{\C}$ on which $\{f^n\}_{n\in \N}$ fails to be a normal family. The complement of the Julia set in $\widehat{\C}$ is called the \textit{Fatou set} and is denoted by $\mathbf F(f)$. See \cite{Mil} for background on Complex Dynamics.

\begin{figure}
\centering
		\begin{tikzpicture}
			\node[] (julia) at (0,0)
				{\includegraphics[width=.6\textwidth]{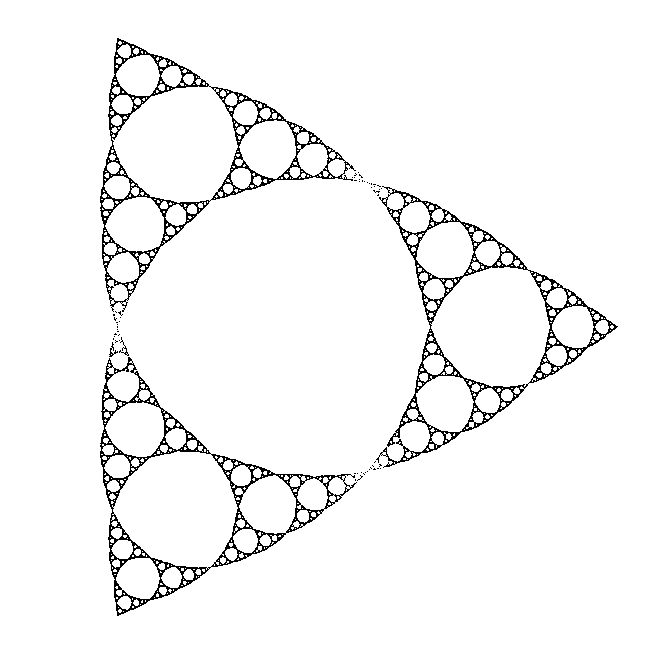}};
		\end{tikzpicture}
		\caption{The Julia set of $f(z)=z^2 -\frac{16}{27z}$. \label{fig:julia}}
\end{figure}

In \cite{DRS} the family of rational maps $f_\lambda(z)= z^2 +\lambda/z^2$, $\lambda\in \C$, is studied. The point $\infty$ lies in the Fatou set, and in fact, there exists a Jordan region $B\subset \widehat{\C}$ containing $\infty$ that lies in the Fatou set. $B$ is called the \textit{basin of attraction of $\infty$}. Furthermore, $f_\lambda$ has four finite critical points of magnitude $|\lambda|^{1/4}$ ($0$ is also a critical point but it is also a pole). One of their results is the following:
\begin{theorem}[{\cite[Theorem 3.1]{DRS}}]\label{Julia Devaney}
If all four critical points of $f_\lambda$ are strictly preperiodic and they lie in $\partial B$, then $\mathbf J (f)$ is a generalized Sierpi\'nski gasket. 
\end{theorem}
A point $z\in \widehat{\C}$ is \textit{preperiodic} if its orbit $\{f^n(z)\}_{n\in \N}$ is a finite set. We say that $z$ is \textit{strictly preperiodic} if $z$ is periodic and $f^n(z)\neq z$ for all $n\in \N$. The notion of a \textit{generalized Sierpi\'nski gasket} is introduced in Section 2 of that paper, and it is a special case of our construction in Proposition \ref{Example Construction}, under some combinatorial assumptions, if we drop the  assumption on the diameters of the sets $V_{i,m}$. 

Assume that $f_\lambda$ is as in the theorem. Translating the setting of \cite{DRS} to our setting, we have $\Omega_0=B$, $K_0=\R^2\setminus \Omega_0$, and $K_m= f_\lambda^{-m}(K_0)$. The  components $V_{i,m}$ of $\inter(K_m)$ are bounded by so-called \textit{$m$-disks}, and it is proved in Proposition 3.6 \cite{DRS} that their diameters converge to $0$. Moreover, under the combinatorial assumption ($4'$) in \cite[p.~22]{DRS}, each boundary $\partial V_{i,m}$ is the union of at most three arcs of sets $\partial \Omega_{j,m'}$ as required in Proposition \ref{Example Construction}. Hence, by Proposition \ref{Example Construction} the Julia set $\mathbf J(f)$  is a detour set. 

In order to conclude that $\mathbf J(f)$ is $W^{1,p}$-removable for $p>2$ it suffices to have that all the Fatou components are uniform H\"older domains. The fact that the critical points of $f$ have finite orbit, or equivalently $f$ is \textit{postcritically finite}, implies that $f$ is \textit{sub-hyperbolic} \cite[Section 19]{Mil}. It is a known fact that the Fatou components of sub-hyperbolic maps are uniform \textit{John domains}, which is strictly stronger than the notion of a H\"older domain; see \cite{SS} for background on John domains, and \cite[Theorem 1]{Mih} for the fact that the Fatou components are John domains.

Another example from Complex Dynamics is given in \cite{Ka}. It is proved there that for self-similar sets $K$ with $d$ similarities such as the Sierpi\'nski gasket there exists a postcritically finite rational map $f$ of degree $d$ whose Julia set is homeomorphic to $K$. In fact there exists a global homeomorphism $h\colon S^2\to S^2$ that maps $\mathbf J(f)$ to $K$. After conjugating $f$ with a M\"obius transformation that maps $h^{-1}(\infty)$ to $\infty$ we may obtain a rational map $\tilde f$ and a homeomorphism $\tilde h \colon \R^2\to \R^2$ that maps $\mathbf J(f)$ to $K$. Remark \ref{Example Homeo invariance} implies that $\mathbf J(\tilde f) $ is a detour set. Also, $\tilde f$ is sub-hyperbolic since it is postcritically finite. It follows that the Fatou components are uniform H\"older domains, and thus $\mathbf J(\tilde f)$ is $W^{1,p}$-removable for $p>2$.

\section{Concluding Remarks}\label{Section Concluding}

As we already remarked in the Introduction, if the Sierpi\'nski gasket $K$ is non-removable for $W^{1,2}$ functions, then there exists a continuous function $f\colon \R^2\to \R$ with $f\in W^{1,2}(\R^2\setminus K)$, but $f\notin W^{1,2}(\R^2)$. By Corollary \ref{Intro Corollary}, we must have 
\begin{align*}
\int |\nabla f|^p= \infty
\end{align*}
for all $p>2$. Of course, this is not a sufficient condition for a function to lie in $W^{1,2}(\R^2\setminus K)\setminus W^{1,2}(\R^2)$. 

In the case of (quasi)conformal removability we obtain an equivalence. 
\begin{prop}\label{Quasiconformal removability}
Let $f\colon \R^2\to \R^2$ be a homeomorphism that is quasiconformal on $\R^2\setminus K$, where $K$ is the Sierpi\'nski gasket. The following are equivalent:
\begin{enumerate}[\upshape(i)]
\item $f$ is quasiconformal on $\R^2$.
\item There exists $p>2$ and and an open ball $B$ containing $K$ with 
\begin{align*}
\int_{B} \|Df\|^p <\infty,
\end{align*}
where $\|Df\|$ denotes the operator norm of the differential of $f$.
\item If $\{\Omega_k\}_{k\geq 1}$ are the bounded complementary components of $K$, then
\begin{align*}
\sum_{k\geq 1} \diam(\Omega_k) \diam(f(\partial \Omega_k))<\infty.
\end{align*}
\end{enumerate}
\end{prop}

\begin{proof}
We first show the equivalence of (i) and (ii). If (ii) holds, then by Corollary \ref{Intro Corollary} we conclude that $f\in W^{1,2}_{\loc}(\R^2)$, which implies that $f$ is quasiconformal on $\R^2$, since $m_2(K)=0$; this follows from the analytic definition of quasiconformality \cite[Definition 2.5.2, p.~24]{AIM}. Conversely, if $f$ is quasiconformal on $\R^2$, then it lies in $W^{1,p}_{\loc}(\R^2)$ for some $p>2$, from which (ii) follows; see \cite{Ge} and also \cite[Theorem 5.1.2]{AIM}.

Next, we justify the equivalence between (i) and (iii). Condition (iii) appeared in the end of the proof of Theorem \ref{Main Theorem} and it suffices to conclude that $f$ is absolutely continuous on almost every line. Hence, $f\in W^{1,2}_{\loc}(\R^2)$ and it is quasiconformal. Conversely, if $f$ is quasiconformal on $\R^2$ then it is \textit{quasisymmetric}, and thus the image of every triangle $ \Omega_k$, $k\geq 1$, contains a ball $B(x_k,r_k)$ and is contained in a ball $B(x_k,R_k)$ with $R_k/r_k \leq C$, where $C>0$ is a constant depending only on $f$; see \cite[Chapters 10--11]{He} for background on quasisymmetric maps and for proof of these claims. Hence,
\begin{align*}
\sum_{k\geq 1}\diam(\Omega_k) \diam(f (\partial \Omega_k)) &\leq \left(\sum_{k\geq 1} \diam(\Omega_k)^2 \right)^{1/2}\left(\sum_{k\geq 1} \diam(f(\partial \Omega_k))^2 \right)^{1/2}\\
&\leq C'\left(\sum_{k\geq 1} m_2(\Omega_k) \right)^{1/2} \left(\sum_{k\geq 1} m_2(B(x_k,r_k)) \right)^{1/2}
\end{align*}
for some constant $C'>0$ depending on $C$. The first sum is finite since it is the sum of the areas of the bounded complementary components of $K$. The second sum is finite, because the balls $B(x_k,r_k)$ are disjoint and are contained in a compact subset of $\R^2$.
\end{proof}


\begin{thebibliography}{}

\bibitem{AIM}
  	Astala K., Iwaniec T., Martin G.:
  	{Elliptic Partial Differential Equations and
Quasiconformal Mappings in the Plane}.
  	Princeton University Press, Princeton, NJ (2009) 

\bibitem{BC}
	Benedetto J., Czaja W.:
	{Integration and modern analysis}.
	Birkh\"auser Advanced Texts: Basler Lehrb\"ucher, Birkh\"auser,
	Boston, Inc., Boston, MA (2009)

\bibitem{Bi}
	Bishop C.:
	{Some homeomorphisms of the sphere conformal off a curve}.
	Ann.\ Acad.\ Sci.\ Fenn.\ Ser.\ A I Math.\ \textbf{19}(2),  323--338 (1994)

\bibitem{BBI}
	Burago D., Burago Y., Ivanov S.:
	{A course in metric geometry}.
	Graduate Studies in Mathematics, vol.\ 33,
	American Mathematical Society,
	Providence, RI (2001)


\bibitem{DRS}
	Devaney R., Rocha M., Siegmund S.:
	{Rational maps with generalized Sierpi\'nski gasket Julia sets}.
	Topology Appl.\ \textbf{154}(1), 11--27 (2007)

\bibitem{Ge}
	Gehring F.W.:
	{The $L^p$-integrability of the partial derivatives of a quasiconformal mapping}.
	Acta Math.\ \textbf{130}, 265--277 (1973)		

\bibitem{He}
	Heinonen J.:
	{Lectures on Analysis on Metric Spaces}.
	Springer Verlag, New York (2001)
	
\bibitem{Jo}
	Jones P.:
	{On removable sets for Sobolev spaces in the plane}.
	Essays on Fourier analysis in honor of Elias M.~Stein (Princeton, NJ, 1991), 250--276,
	Princeton Math.\ Ser., vol.\ 42, Princeton Univ.\ Press, Princeton, NJ (1995)

\bibitem{JS}
	Jones P., Smirnov S.:
	{Removability theorems for Sobolev functions and quasiconformal maps}.
	Ark.\ Mat.\ \textbf{38}(2), 263--279 (2000)

\bibitem{Ka}
	Kameyama A.:
	{Julia sets of postcritically finite rational maps and topological self-similar sets}.
	Nonlinearity \textbf{13}(1), 165--188 (2000)

\bibitem{Kau}
	Kaufman R.:
	{Fourier-Stieltjes coefficients and continuation of functions}.
	Ann.\ Acad.\ Sci.\ Fenn.\ Ser.\ A I Math.\ \textbf{9}, 27--31 (1984)

\bibitem{Ko}
	Koskela P.:
	{Removable sets for Sobolev spaces}.
	Ark.\ Mat.\ \textbf{37}(2), 291--304 (1999)

\bibitem{KRZ}
	Koskela P., Rajala T., Zhang Y.:
	{A density problem for Sobolev functions on Gromov hyperbolic domains}.
	Nonlinear Anal.\ \textbf{154}, 189--209 (2017)

\bibitem{KN}
	Koskela P., Nieminen T.:
	{Quasiconformal removability and the quasihyperbolic metric}.
	Indiana Univ.\ Math.\ J.\ \textbf{54}(1), 143--151 (2005)

\bibitem{MV}
	Martio O., V\"ais\"al\"a J.:
	{Quasihyperbolic geodesics in convex domains II}.
	Pure Appl.\ Math.\ Q.\ \textbf{7}(2),
	Special Issue: In honor of Frederick W.~Gehring, Part 2, 395--409 (2011)

\bibitem{Maz}
	Maz'ya V.:
	Sobolev Spaces with Applications to Elliptic Partial Differential Equations.
	Grundlehren der mathematischen Wissenschaften, vol.\ 324,
	Springer-Verlag Berlin Heidelberg (2011)
	
\bibitem{Mih}
	Mihalache N.:
	{Julia and John revisited}.
	Fund.\ Math.\ \textbf{215}(1), 67--86 (2011)

\bibitem{Mil}
	Milnor J.:
	{Dynamics in one complex variable}.
	Third edition, Annals of Mathematical Studies, vol.\ 160, 
	Princeton University Press, Princeton, NJ (2006)

\bibitem{New}
	Newman M.~H.~A.:
	{Elements of the topology of plane sets of points}.
	Second edition,
	Cambridge University Press, London,
	1964.

\bibitem{NtaCarpet}
	Ntalampekos, D.:
	Non-removability of Sierpi\'nski carpets.
	Preprint arXiv:1809.05605 (2018)

\bibitem{Nta}
	Ntalampekos, D.:
	Non-removability of the Sierpi\'nski gasket.
	Invent.\ Math.\ \textbf{216}(2), 519-595 (2019)

\bibitem{NtaWu}
	Ntalampekos, D., Wu J.-M.:
	Non-removability of Sierpi\'nski spaces.
	Proc. Amer. Math. Soc., to appear, \url{https://doi.org/10.1090/proc/14698} (2019) 

\bibitem{Sh}
	Sheffield S.:
	{Conformal weldings on random surfaces: SLE and the quantum gravity zipper}.
	Ann.\ Probab.\ \textbf{44}(5), 3474--3545 (2016)
	
	
\bibitem{SS}
	Smith W., Stegenga D.:
	{H\"older domains and Poincar\'e domains}.
	Trans.\ Amer.\ Math.\ Soc.\ \textbf{319}(1), 67--100 (1990)	

\bibitem{Va}
	V\"ais\"al\"a J.:
	{Lectures on $n$-dimensional quasiconformal mappings}.
	Lecture Notes in Mathematics, vol.\ 229, 
	Springer-Verlag, Berlin-New York (1971)

\bibitem{Wh}
	Whyburn G.T.:
	{Analytic Topology}.
	American Mathematical Society Colloquium Publications, vol.\ 28,
	American Mathematical Society, New York (1942)

\bibitem{Yo}
	Younsi M.:
	{On removable sets for holomorphic functions}.
	EMS Surv.\ Math.\ Sci.\ \textbf{2}(2), 219--254 (2015)
	
\end{thebibliography}
\end{document}